\documentclass{amsart}
\usepackage{graphicx}
\usepackage{graphicx,esint}
\usepackage{amssymb,amscd,amsthm,amsxtra}
\usepackage{latexsym}
\usepackage{epsfig}

\newtheorem{thm}{Theorem}[section]

\newtheorem{lem}[thm]{Lemma}
\newtheorem{prop}[thm]{Proposition}
\theoremstyle{definition}

\theoremstyle{remark}

\numberwithin{equation}{section}

\newcommand{\R}{\mathbb R}
\newcommand{\e}{\textit{\textbf{e}}}
\newcommand{\ep}{\eps}

\newcommand{\eps}{\varepsilon}

\newcommand{\p}{\partial}

\newcommand{\comment}[1]{}

\begin{document}

\title{Minimizers of convex functionals arising in random surfaces}%
\author{D. De Silva}
\address{Department of Mathematics, Barnard College, Columbia University, New York, NY 10027}
\email{\tt  desilva@math.columbia.edu}
\author{O. Savin}%
\address{Department of Mathematics, Columbia University, New York, NY 10027 }%
\email{\tt  savin@math.columbia.edu}%
\footnote{The second author was supported by N.S.F. Grant
DMS-07-01037 and a Sloan Fellowship.}


\begin{abstract} We investigate $C^1$ regularity of minimizers to
$\int F(\nabla u)dx$ in two dimensions for certain classes of
non-smooth convex functionals $F$. In particular our results apply
to the surface tensions that appear in recent works on random
surfaces and random tilings of Kenyon, Okounkov and others.
\end{abstract}
\maketitle

\section{Introduction}

The classical problem in the calculus of variations consists in
minimizing the functional $$\min_{\mathcal A}\int_\Omega F(\nabla
u)dx, \quad \quad \Omega\subset \mathbb{R}^n$$ where $F:\mathbb R
^n \to \mathbb R$ is a given convex function and $\mathcal A$ is a
family of admissible functions $u:\Omega \to \mathbb{R}$
(typically the elements of a functional space satisfying a
boundary condition $u=\varphi$ on $\partial \Omega$). In general
it is not hard to show the existence and uniqueness for this
problem, but the main difficulty arises when trying to prove
further smoothness properties of the minimizers. For example, one
would like to show that a minimizer has continuous second
derivatives. Then it satisfies the Euler-Lagrange equation
\begin{equation}{\label{EL`}}
div(\nabla F(\nabla u))=F_{ij}(\nabla u)u_{ij}=0.
\end{equation}

Hilbert's 19th problem posed at the beginning of the last century
refers precisely to this question of regularity: {\it Are the
solutions of regular problems in the calculus of variations always
necessarily analytic?}

In two dimensions the problem was solved in 1943 by C.B. Morrey
\cite{M} who showed analyticity of solutions by the use of complex
analysis and quasiconformal mappings. A partial answer to
Hilbert's problem in the general case was given in the 1930s by
the use of Schauder's estimates for linear equations
\cite{Sc1},\cite{Sc2}, which guarantee that minimizers are smooth
once they have H\"older continuous derivatives. On the other hand
it follows from the comparison principle that minimizers are
Lipschitz continuous under quite general conditions on the domain
and boundary data. One piece of the puzzle remained to be proved:
that Lipschitz minimizers are in fact $C^{1, \alpha}$.

The breakthrough came in the 1950s from the work of E. De Giorgi
\cite{D} on the minimal surface equation and independently from J.
Nash \cite{N}. Differentiating the equation (\ref{EL`}) with
respect to $x_k$ we see that the derivative $u_k$ satisfies the
nonlinear elliptic PDE
\begin{equation}{\label{1`}}
\partial_i(F_{ij}(\nabla u)u_{kj})=0.
\end{equation}
Ignoring the dependence on $\nabla u$ we can write that $v=u_k$ satisfies the linear equation
$$\partial_i(a_{ij}(x)v_{j})=0.$$
De Giorgi observed that the uniform ellipticity condition
$$\lambda I \le [a_{ij}(x)]_{i,j} \le \Lambda I$$
on the coefficients is sufficient to obtain continuity for $v$ (in
fact H\"older continuity due to scaling). The proof is based on an
iteration scheme for various Caccioppoli type inequalities of the
form
\begin{equation}{\label{cacc}}
\int_\Omega |\nabla v|^2\eta^2dx \le C(\Lambda, \lambda)
\int_\Omega v^2|\nabla \eta|^2dx, \quad \quad \quad \forall \eta
\in C_0^\infty(\Omega).
\end{equation}
This result led to extensive study of the theory of linear second
order elliptic equations with measurable coefficients (see for
example \cite{CFMS},  \cite{GT}, \cite{LSW}).

With the question of regularity being understood in the case when
$F$ is smooth and strictly uniformly convex, it is natural to
further investigate what happens when the uniform ellipticity
condition on $D^2F$ fails on a certain set. In the particular case
of the $p$-Laplace equation, i.e $F(\xi)=|\xi|^p$, $1 < p <
\infty$, when the degeneracy set consists of only one point (the
origin)  the regularity of minimizers is well-understood (see fore
example \cite{E}, \cite{U}).

In this paper we direct our attention to the case when the
degeneracy set of $F$ can possibly be very large. The first step
towards proving regularity of minimizers would be obtaining $C^1$
continuity. Then the Euler Lagrange equation becomes less
ambiguous and the classical theory can be applied near the points
whose gradients lie outside the degeneracy set of $F$.

If $F^*$ denotes the Legendre transform of $F$ then
$$F_{ij}(\nabla F^*)F^*_{ij}=tr((D^2F^*)^{-1}D^2F^*)=n,$$
thus $F^*$ is a particular solution of the nonlinear PDE
(\ref{1`}). Without further evidence we can ask wether or not
minimizers have the same regularity as $F^*$. In particular: {\it
Is it true that if $F$ is strictly convex then Lipschitz
minimizers are of class $C^1$?} This seems to be a difficult
question but there is some evidence that suggests that the result
might be true at least in two dimensions.

In the first part of the paper we prove {\it a priori} estimates
that answer the question above in two dimensions for two large
classes of strictly convex functionals. Our first theorem shows
that if $\mathbb R^2$ can be covered by two open sets $O_\lambda$,
$V_\Lambda$,
\begin{align*}
O_\lambda& \subset \{p \in \R^2, D^2F(p) > \lambda I\},\\
V_\Lambda& \subset \{p \in \R^2, D^2F(p) < \Lambda I\},
\end{align*}
then $\nabla u$ has a uniform modulus of continuity in the interior of $\Omega$.

\begin{thm}\label{Main} Let $u$ be a minimizer with
$\|\nabla u\|_{L^\infty(B_1)} \leq M.$ Assume that
$$\overline{\mathcal{B}}_M \subset O_\lambda \cup V_\Lambda.$$ Then in
$B_{1/2}$, $\nabla u$ has a uniform modulus of continuity
depending on the modulus of convexity $\omega_F$ of $F$,
$O_\lambda, V_\Lambda, M$ and $ \|\nabla
F\|_{L^{\infty}(\mathcal{B}_{M})}.$
\end{thm}

In the second theorem we show that if $D^2 F>0$ except at a finite number of points then the same result holds.

\begin{thm}\label{Main2} Let $u$ be a minimizer with
$\|\nabla u\|_{L^\infty(B_1)} \leq M.$ Assume that $\overline
{\mathcal{B}}_M \setminus \bigcup O_{1/n}$ is a finite set. Then
in $B_{1/2}$, $\nabla u$ has a uniform modulus of continuity
depending on $O_{1/n}, M$ and $ \|\nabla
F\|_{L^{\infty}(\mathcal{B}_{M})}.$
\end{thm}

Our interest in this two dimensional regularity problem is motivated by a series of recent papers in combinatorics and statistical mechanics about random tilings and random surfaces by Cohn, Kenyon, Okounkov, Propp, Sheffield and others.
Let us briefly explain the connection between our problem and these results.

We present the simplest model of a random surface. \noindent For
small $\varepsilon$ consider the points $\varepsilon \mathbb{Z}^3
\subset \mathbb{R}^3$, and call the $(1,1,1)$ direction vertical
and the plane $P=\{x+y+z=0\}$ horizontal. An {\it
$\varepsilon$-stepped surface} is a polygonal surface whose faces
are squares in the 2-skeleton of $\varepsilon \mathbb{Z}^3$ and
which is a graph in the vertical direction. In other words, the
subgraph of an $\varepsilon$-stepped surface is a collection of
$\varepsilon$-cubes of $\varepsilon \mathbb{Z}^3$.

Obviously, stepped surfaces can approximate only Lipschitz graphs with
gradients lying in the equilateral triangle $N$ with vertices given by the slopes of
the three planes $\{x=0\}$, $\{y=0\}$, $\{z=0\}$.

Consider the ``random surfaces" given by all $\varepsilon$-
stepped surfaces above a domain $\Omega \subset P$ which stay
$O(\varepsilon)$ away from a given function $\varphi$ defined on
$\partial \Omega$. In connection with problems that arise in
statistical mechanics, it is interesting to investigate the
limiting behavior as $\varepsilon \to 0$. Cohn, Kenyon and Propp
\cite{CKP} proved that as $\varepsilon \to 0$ the random surfaces
almost surely converge uniformly to a nonrandom function $u$ which
solves the variational problem
\begin{equation}{\label{100}}
\min_u\int_\Omega F(\nabla u)dx, \quad \nabla u \in \overline N, \quad
\mbox{$u=\varphi$ on $\partial \Omega$.}
\end{equation}
The surface tension $F:\overline N\to \mathbb{R}$ can be computed
from combinatorial considerations and is strictly convex in the
interior of $N$ and constant on $\partial N$.

The picture of such a random surface is shown in Figure 1 and it
appeared on the cover of the Notices of the AMS, March 2005,
Volume 52.

\begin{figure}
\centering \scalebox{0.4}{
        \epsfig{file=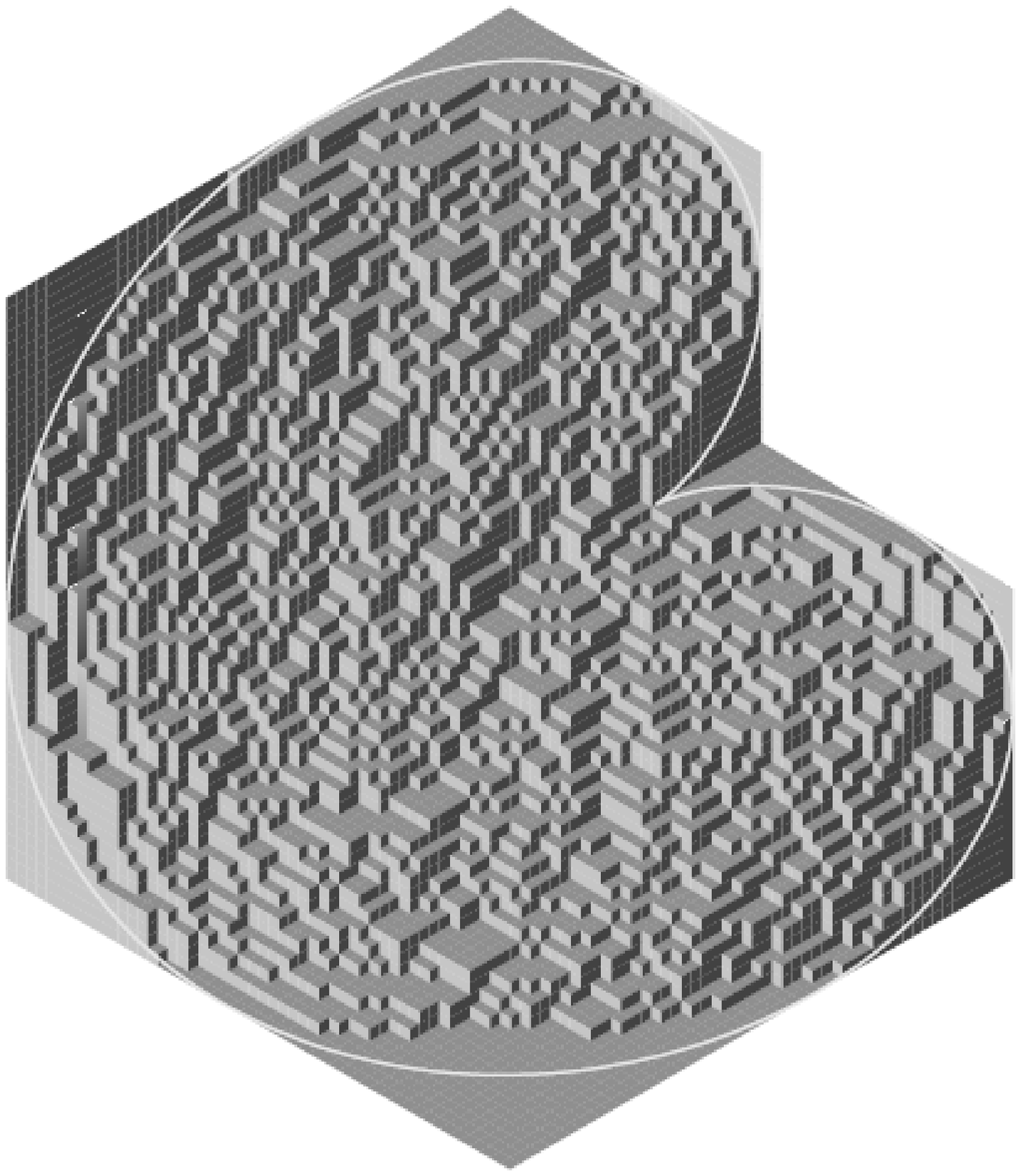}
} \caption{ \ } \label{fig1}
\end{figure}

Other types of random surfaces can be obtained from random tiling
of the plane with various geometric figures. For example there is
a one to one correspondence between the tilings with $60^o$ rhombi
of an $\varepsilon$- honeycomb lattice and the
$\sqrt{3}\varepsilon$- stepped surfaces: the tilings are the
projections of the faces (squares) along the vertical direction
onto the plane $P$.


More generally, one can construct random surfaces corresponding to
random perfect matchings (also called {\it dimer configurations})
on a weighted, bipartite, periodic, planar graph $G$. Kenyon,
Okounkov and Sheffield \cite{KOS} showed that these random
surfaces converge {\it a.s.} to a solution of (\ref{100}) with $F$
and the polygon $N$ (called the Newton polygon) depending on $G$.
The function $F$ can be explicitly computed and minimizers of
\eqref{100} for this $F$ and for a special
class of domains and boundary
data were studied by Kenyon and Okounkov \cite{KO} using techniques from
algebraic geometry.

In particular $F$ is piecewise linear on $\partial N$, strictly
convex in the interior, smooth except at a finite number of points
where it can have singularities.

In the second part of this paper we investigate precisely the
$C^1$ regularity of minimizers of (\ref{100}) for an arbitrary 
functional $F$ which is, in the interior of a polygon $N$, strictly
convex and smooth except at a finite number of points, but without
any further assumptions about the behavior of $F$ on $\partial N$.
In this case the set of degeneracy of $D^2F$ is the union of a
finite set with $\partial N$.

 The variational
problem (\ref{100}) is equivalent to an obstacle problem. Indeed,
let $\overline \varphi$ ($\underline \varphi$) be the minimum
(maximum) of all admissible functions $u$ and define $F$ to be
$\infty$ outside $\overline N$, then problem (\ref{100}) is
equivalent to

$$ \min_{\underline \varphi \le u \le \overline \varphi}\int_\Omega F(\nabla u)dx.$$
 The obstacle problem has been extensively studied for the Laplace equation i.e $F(\xi)=|\xi|^2$ (see for example \cite{C}).

 The regularity of minimizers in our case is quite delicate. Near the points where
the solution separates from the obstacle the equation becomes
degenerate. Moreover, even in the set where $\underline
\varphi<u<\overline \varphi$ it cannot be concluded (as in Laplace
equation case) that $\nabla u$ is in the interior of $N$. Thus the
results from the first part of the paper cannot be directly
applied in any reasonable set.

 Our main regularity result says that the minimizers
are $C^1$ in $\Omega$ except on a number of segments which have an
end point on $\partial \Omega$ and have directions perpendicular
to the sides of $N$. On these segments the minimizer coincides
with either the lower or the upper obstacle.

\begin{thm}\label{C1} If $\nabla u$ is discontinuous at $x_0 \in  \Omega$,
then there exists a direction $\nu_i$ perpendicular to one of the
sides $[p_i,p_{i+1}]$ of $N$ such that $u$ is linear on a segment of
direction $\nu_i$ connecting $x_0$ and $\partial \Omega$.
Precisely, there exists $x_1=x_0 + t_1 \nu_i \in \partial \Omega$
such that
$$u(x) = u(x_0)+ p_i \cdot (x-x_0), \quad \text{for all} \ x \in [x_0,x_1] \subset \overline{\Omega}.$$
In particular, $u \in C^1$ away from the obstacles.\end{thm}

We also prove a type of continuity result at the points where the
minimizer is not differentiable. We show that if a sequence of
points converges to a point of non-differentiability then their
corresponding gradients approach $\partial N$ (see Theorem
\ref{continuous}).

As seen in Figure 1, minimizers of (\ref{100}) may have ``flat"
regions where they are linear and their gradient belongs to the
set of vertices of $N$. We show that these regions must lie
between a convex and a concave graph (see Theorem \ref{convex}).

The proof of our results relies on two dimensional techniques. For
the results in the first part we prove weak versions of
Caccioppoli inequalities in which the right hand side of
(\ref{cacc}) is replaced by a constant. This is done using the
precise form of the nonlinearity in equation (\ref{1`}). As it
turns out, such inequalities are sufficient for proving $C^1$
continuity. The second part of the paper is more involved. We use
the ideas of the first part and an approximation technique
together with the fact that any tangent plane splits the graph of
a smooth minimizer into at least four connected components. These
methods could be exploited further to include other examples of
functionals, however most examples would satisfy our assumptions.

 As in the work of De Giorgi
the key step in proving $C^1$ continuity consists in obtaining
``localization" lemmas (see Lemmas \ref{A1}, \ref{A2}, \ref{On},
and Theorem \ref{localization}). Roughly speaking, they say that
when we restrict to smaller and smaller domains the image of the
gradient restricts itself to either one of two known sets (that
decompose the plane).

The paper is organized as follows. In Section 2 and Section 3 we
prove respectively Theorem \ref{Main} and Theorem \ref{Main2}. In
Section 4 we present the degenerate obstacle problem and state the
main results. In Section 5 we introduce the approximation problem.
In Section 6 we prove Theorem \ref{C1}.  In Section 7 we prove
Theorem \ref{continuous}. Finally in the last section we prove two
flatness theorems Theorems \ref{Flat1} and \ref{Flat2} that are
used in the course of the proof of Theorem \ref{C1}.

\section{The proof of Theorem \ref{Main}.}

Let $F : \R^2 \rightarrow \R$ be a smooth strictly convex
function, with modulus of convexity $\omega_F: (0,+\infty)
\rightarrow (0,+\infty)$, i.e.

$$F(q) - F(p) - \nabla F(p) \cdot (q-p) \geq \omega_F(|q-p|), \quad q \neq p.$$

Let $0<\lambda \leq \Lambda,$ and denote by $O_\lambda, V_\Lambda$
two open sets such that
\begin{align}
O_\lambda& \subset \{p \in \R^2, D^2F(p) > \lambda I\},\\
V_\Lambda& \subset \{p \in \R^2, D^2F(p) < \Lambda I\}.
\end{align}

Throughout this paper balls (of radius $r$) in the $x$-variable
space are denoted by $B_r$, while balls in the $p$-variable space
(the gradient space), are denoted by $\mathcal{B}_r.$

Consider the integral functional
\begin{equation*}
I(u)= \int_{B_1} F(\nabla u)dx.
\end{equation*}
We prove the following a priori estimate.

\

\noindent \textbf{Theorem 1.1.} \textit{Let $u$ be a minimizer to
$I(u)$ with $\|\nabla u\|_{L^\infty(B_1)} \leq M.$ Assume that
$$\overline{\mathcal{B}}_M \subset O_\lambda \cup V_\Lambda.$$ Then in
$B_{1/2}$, $\nabla u$ has a uniform modulus of continuity
depending on $\omega_F, O_\lambda, V_\Lambda, M$ and $ \|\nabla
F\|_{L^{\infty}(\mathcal{B}_{M})}.$}

\subsection{Statement of the localization Lemmas.} The proof of Theorem
1.1 relies on the next two Lemmas. Let $\textit{\textbf{e}}$ be an
arbitrary unit vector (direction) in $\R^2$ and denote by
\begin{align*}
&H^+_\e(c):= \{p \in \R^2 | p \cdot \textit{\textbf{e}} > c\},\\
&H^-_\e(c):= \{p \in \R^2 | p \cdot \textit{\textbf{e}} <c\},\\
&S_{\textit{\textbf{e}}}(c_0,c_1):=\{p \in \R^2 | c_0 < p\cdot
\textit{\textbf{e}}<c_1\},
\end{align*}
for some constants $c,c_0,c_1 \in \R,$ with $c_0 <c_1.$

\begin{lem}\label{A1}Let $u$ be a minimizer to $I(u)$ with
$\|\nabla u\|_{L^\infty(B_1)} \leq M.$ Assume that there exist a
direction $\textit{\textbf{e}}$ and constants $c_0,c_1$ such that
\begin{equation}\label{inclusion}S_{\textit{\textbf{e}}}(c_0,c_1) \cap \nabla u (B_1) \subset
O_\lambda.\end{equation}Then, there exists $\delta>0$ depending on
$c_1-c_0, \lambda,M, \|\nabla F\|_{L^{\infty}(\mathcal{B}_{M})},$
such that either
\begin{equation}\label{upperplane}
\nabla u(B_\delta) \subset H^+_\e(c_0)
\end{equation}or
\begin{equation}\label{lowerplane}
\nabla u(B_\delta) \subset H^-_\e(c_1).
\end{equation}
\end{lem}

\begin{lem}\label{A2}Let $u$ be a minimizer to $I(u)$ with
$\|\nabla u\|_{L^\infty(B_1)} \leq M.$  Assume that there exist a
direction $\textit{\textbf{e}}$ and constants
$\tilde{c},\varepsilon$ such that
\begin{equation}\label{inclusion2} H^+_\e(\tilde{c}-\eps) \cap \nabla u (B_1) \subset
V_\Lambda.\end{equation}Then, there exists $\delta>0$ depending on
$\eps, \Lambda, \omega_F, M$, such that either
\begin{equation}\label{upperplane2}
\nabla u(B_\delta) \subset H^+_\e(\tilde{c}-\eps)
\end{equation}or
\begin{equation}\label{lowerplane2}
\nabla u(B_\delta) \subset H^-_\e(\tilde{c}+\eps).
\end{equation}
\end{lem}

As observed in the introduction, these Lemmas say that as we
restrict to smaller and smaller domains in the $x$-space the image
of the gradient restricts itself to either one of two half-planes.
The Caccioppoli-type inequalities in the next subsection will be
the key tool towards the proof of these Lemmas.

\subsection{Caccioppoli-type inequalities.}
Let $u$ be a smooth minimizer to $I(u)$, then $u$ satisfies the
Euler-Lagrange equation
\begin{equation}\label{EL}
\text{div}(\nabla F(\nabla u))= 0 \ \text{in $B_1$}.
\end{equation} Differentiating the equation in the direction $\e$ we obtain
\begin{equation}\label{weak}
\int_{B_1} F_{ij} (\nabla u) u_{\e j} \phi_i dx =0, \ \forall \phi
\in C_0^\infty(B_1).
\end{equation}

Let $c_0,c_1 \in \R$ with  $c_0 <c_1$ and denote by
\begin{equation}\label{G}
G(t):=
  \begin{cases}
    0 & \text{$t \leq c_0$}, \\
    1 & \text{$t \geq c_1$},
  \end{cases}
\end{equation}
a smooth function with bounded slope $|G'(t)| < C$.

\begin{prop}\label{Cacc1}Let $u$ be a minimizer to $I(u)$ with
$\|\nabla u\|_{L^\infty(B_1)} \leq M.$  Assume that there exist a
direction $\textit{\textbf{e}}$ and constants $c_0,c_1$ such that
\begin{equation}\label{inclusion3}S_{\textit{\textbf{e}}}(c_0,c_1) \cap \nabla u (B_1) \subset
O_\lambda.\end{equation}Then,
\begin{equation}\label{Caccineq1}
\int_{B_{1/2}} |\nabla (G(u_{\textit{\textbf{e}}}))|^2 dx \leq C
\end{equation}
for some constant $C$ depending on $c_0,c_1,\lambda, M, \|\nabla
F\|_{L^{\infty}(\mathcal{B}_{M})}$.
\end{prop}
\begin{proof} In formula \eqref{weak}, let us choose as test function $\phi= \xi^2
G(u_\textit{\textbf{e}})$, with $\xi \in C^\infty_0(B_1), 0 \leq
\xi \leq 1$ and $\xi \equiv 1$ on $B_{1/2}$. We obtain
\begin{equation}\label{step1}
\int_{B_1} F_{ij}(\nabla u)u_{\e j}\xi^2(G(u_\e))_i dx = -2
\int_{B_1} F_{ij}(\nabla u) u_{\e j} \xi \xi_i G(u_\e) dx.
\end{equation}Now we analyze the left-hand side of \eqref{step1}.

\begin{align}
\nonumber LHS&= \int_{B_1} F_{ij}(\nabla u)u_{\e j}u_{\e i} \xi^2
G'(u_\e) dx \\ \label{defA1}& \geq \lambda \int_{B_1} |\nabla
u_\e|^2 \xi^2 G'(u_\e)dx \geq c\lambda \int_{B_1} |\nabla
(G(u_\e))|^2 \xi^2 dx,\end{align} where in the first inequality in
\eqref{defA1} we used the assumption \eqref{inclusion3} and the
definition of $O_\lambda$, while in the second one we used that
$G'$ is bounded. Thus,
\begin{equation}\label{LHS}
LHS \geq c\lambda \int_{B_{1}}|\nabla (G(u_\e))|^2\xi^2 dx.
\end{equation} On the other hand, using integration by parts we
have that the right-hand side of \eqref{step1} is given by
\begin{align*}
RHS  & = - 2  \int_{B_1} \partial_\e(F_i(\nabla u))\xi\xi_i
G(u_\e)dx \\\label{parts} & = 2 \int_{B_1} F_i(\nabla u)
\partial_\e (\xi \xi_i) G(u_\e)dx +
\int_{B_1} F_i(\nabla u)\xi \xi_i \partial_\e
(G(u_\e))dx.\end{align*} Thus, for any $\gamma \leq 1$,
\begin{equation} \label{young}
|RHS| \leq \gamma\int_{B_1} |\nabla (G(u_\e))|^2 \xi^2 dx +
C/\gamma
\end{equation}
with $C$ depending on $\|\nabla F\|_{L^\infty(\mathcal{B}_M)}.$
Combining \eqref{LHS} with \eqref{young} and choosing $\gamma$
sufficiently small we obtain the desired claim.

\end{proof}

\begin{prop}\label{Cacc2}Let $u$ be a minimizer to $I(u)$ with
$\|\nabla u\|_{L^\infty(B_1)} \leq M.$ Assume that there exist a
direction $\textit{\textbf{e}}$ and constants $\tilde{c}, \eps $
such that
\begin{equation}\label{inclusion4} H^+_\e(\tilde{c}-\eps) \cap \nabla u (B_1) \subset
V_\Lambda.\end{equation}Then, set $w=\nabla F(\nabla u)$ and $U=
(\nabla u)^{-1}(H^+_\e(\tilde{c})),$
\begin{equation}\label{Caccineq2}
\int_{B_{1/2} \cap U} |\nabla w|^2 dx \leq C
\end{equation}
for some constant $C$ depending on $\tilde{c}, \eps, \Lambda, M$.
\end{prop}
\begin{proof}In formula \eqref{weak} choose as test function $\phi= \xi^2(u_\e -
(\tilde{c}-\eps))^+$ with $\xi \in C_0^\infty(B_1), 0 \leq \xi
\leq 1$ and $\xi \equiv 1$ on $B_{1/2}.$ Set $U^\eps := (\nabla
u)^{-1}(H^+_\e(\tilde{c}-\eps)) \supset U.$

Then,
\begin{equation}\label{test2}
\int_{U^\eps} F_{ij}(\nabla u)u_{\e j}u_{\e i}\xi^2dx = -2
\int_{B_1}F_{ij}(\nabla u)u_{\e j}\xi \xi_i (u_\e -
(\tilde{c}-\eps))^+dx.
\end{equation}
According to assumption \eqref{inclusion4} and the definition of
$V_\Lambda$, on the set $U^\eps$ we have that
\begin{align*}
F_{ij}(\nabla u)u_{\e j}u_{\e i} &= (\nabla u_\e)^T \cdot
D^2F(\nabla u) \cdot \nabla u_\e \geq
\frac{1}{\Lambda}|D^2F(\nabla u)\cdot \nabla u_\e|^2.
\end{align*}
Hence, after bounding the right-hand side of \eqref{test2} using
Young's inequality we obtain
\begin{equation}\label{test3}
\int_{U^\eps} |D^2F(\nabla u)\cdot \nabla u_\e|^2 \xi^2 dx \leq
\gamma \int_{U^\eps} |D^2F(\nabla u)\cdot \nabla u_\e|^2 \xi^2 dx
+ C/\gamma
\end{equation}
for some constant $C$ depending on $\Lambda, M,\tilde{c},\eps$ and
any $\gamma
>0.$ Hence for $\gamma$ small enough we obtain
\begin{equation}\label{test4}
\int_{B_{1/2} \cap U^\eps} |D^2F(\nabla u)\cdot \nabla u_\e|^2 dx
\leq C.
\end{equation}
Now, let $\textbf{f}$ be a direction close to $\e$, $\textbf{f}
\neq \e$, such that for some constant $k \in \R,$
$$\{p \cdot \textbf{f} = k \} \cap \mathcal{B}_M \subset S_\e(\tilde{c}-\eps,
\tilde{c}).$$ Then,
\begin{equation*} H^+_\textbf{f}(k) \cap \nabla u (B_1) \subset
V_\Lambda,\end{equation*} and we can repeat the same argument as
above with $U^k := (\nabla u)^{-1}(H^+_\textbf{f}(k)) \supset U$
to conclude that
\begin{equation}\label{test5}
\int_{B_{1/2} \cap U^k} |D^2F(\nabla u)\cdot \nabla
u_{\textbf{f}}|^2 dx \leq C.
\end{equation}
Combining \eqref{test4} with \eqref{test5} we obtain
\begin{equation*}
\int_{B_{1/2} \cap U} |D^2F(\nabla u)D^2 u|^2 dx \leq C,
\end{equation*}
which gives the desired inequality.

\end{proof}

\subsection{The proof of the localization Lemmas.}

We are now ready to prove the localization Lemmas, using the
Caccioppoli-type inequalities above and the fact that our problem
is two dimensional.

\

\noindent \textbf{Proof of Lemma \ref{A1}.} Let $\delta>0$ and
assume that there exist $x_0,x_1 \in B_\delta$ such that $$\nabla
u (x_0) \cdot \e \leq c_0, \quad \nabla u (x_1) \cdot \e \geq
c_1.$$ Then, by the maximum principle
$$ \{\nabla u \cdot \e \leq c_0\} \cap \partial B_r \neq \emptyset,$$
$$ \{\nabla u \cdot \e \geq c_1\} \cap \p B_r \neq \emptyset,$$
for all $r$, with $\delta \leq r \leq 1/2.$
Let $$x_0^r \in  \{\nabla u \cdot \e \leq c_0\} \cap
\partial B_r, \quad x_1^r \in  \{\nabla u \cdot \e \geq c_1\} \cap
\partial B_r$$ for all $\delta \leq r \leq 1/2.$ Then, if
$G$ is the function defined in \eqref{G}, we have
\begin{equation}\label{model}
1 = G(u_\e(x_1^r)) - G(u_\e(x_0^r)) \leq \int_{\partial
B_r(x_0^r,x_1^r)} |\nabla (G(u_\e))|ds.
\end{equation}
Applying Cauchy-Schwartz we obtain
\begin{equation*}
\int_{\partial B_r(x_0^r,x_1^r)} |\nabla (G(u_\e))|^2ds \geq
\frac{1}{2\pi r}.
\end{equation*}Thus,
\begin{equation*}
\int_{\delta}^{1/2} \int_{\partial B_r(x_0^r,x_1^r)} |\nabla
(G(u_\e))|^2ds dr \geq
\frac{1}{2\pi}\int_{\delta}^{1/2}\frac{1}{r}dr =\frac{1}{2\pi}\ln
\frac{1}{2\delta}.
\end{equation*}Hence,
\begin{equation*}
\int_{B_{1/2}} |\nabla (G(u_\e))|^2dx \geq \frac{1}{2\pi}\ln
\frac{1}{2\delta},
\end{equation*}
which for $\delta$ sufficiently small depending on $c_0, c_1,
\lambda,M, \|\nabla F\|_{L^\infty(\mathcal{B}_M)},$ contradicts
the inequality \eqref{Caccineq1}.

\qed

\

\noindent \textbf{Proof of Lemma \ref{A2}.} Let $\delta>0$ and
assume that there exist $x_0,x_1 \in B_\delta$ such that $$\nabla
u (x_0) \cdot \e \leq \tilde{c}, \quad \nabla u (x_1) \cdot \e
\geq \tilde{c} + \eps.$$ Then, as in the previous lemma, for all
$r$ with $\delta \leq r \leq 1/2$ there are points $x_0^r,x_1^r
\in
\partial B_r$ such that
$$x_0^r \in \{\nabla u  \cdot \e \leq \tilde{c}\}, \quad x_1^r \in \{\nabla u \cdot \e \geq \tilde{c} + \eps\}.$$
 Denote by
$$\underline{H}:= \nabla F (\{p \cdot \e \leq \tilde{c}\})$$
$$\overline{H} : = \nabla F(\{p \cdot \e \geq \tilde{c} + \eps \}).$$
Set $w= \nabla F (\nabla u)$, then $w$ maps $x_0^r$ and $x_1^r$
respectively in $\underline{H}$ and $\overline{H}$.

Now, since $F$ is strictly convex, there exists
$\eta=\eta(\eps)>0$ depending on $\omega_F$ such that
$$|p_1 - p_2| \geq \eps \Rightarrow |\nabla F(p_1) -\nabla F(p_2)| \geq 2\eta.$$
Let $\mathcal{G}$ be a smooth function on $\R^2$ with $|\nabla
\mathcal{G}| \leq 1$ such that $$\mathcal{G}(p)=
  \begin{cases}
    0 & \text{$p \in \underline{H}$}, \\
    \eta & \text{$p \in \overline{H}$}.
  \end{cases}
$$ Notice that $\mathcal{G}$ can be obtained by a mollification of
$\min\{\eta, (dist(p,\underline{H}) - \gamma)^+\}$, with $\gamma$
small.

Then,
 we have
\begin{equation}
\eta = \mathcal{G}(w(x_1^r)) - \mathcal{G}(w(x_0^r)) \leq
\int_{\partial B_r(x_0^r,x_1^r)} |\nabla (\mathcal{G}(w))|ds.
\end{equation}


Arguing as in the proof of the previous Proposition (see
computations following \eqref{model}), we get
\begin{equation}\label{semifinal} \int_{B_{1/2}} |\nabla
(\mathcal{G}(w))|^2dx \geq \frac{\eta}{2\pi}\ln \frac{1}{2\delta}.
\end{equation}

However,
\begin{align}\label{final}
\int_{B_{1/2}} |\nabla (\mathcal{G}(w))|^2dx &\leq \int_{B_{1/2}}
|\nabla \mathcal{G}|^2 |\nabla w|^2dx \leq \int_{B_{1/2} \cap
(\nabla u)^{-1}(S_\e(\tilde{c},\tilde{c}+\eps))} |\nabla w|^2dx.
\end{align}
Therefore, combining \eqref{semifinal} with \eqref{final}, we
obtain a contradiction to \eqref{Caccineq2} as long as $\delta$ is
sufficiently small.

\qed


\subsection{The proof of Theorem 1.1.} We wish to prove the
following claim:
$$\forall \eps >0, \exists \delta>0, \delta=\delta(\ep,\omega_F, O_\lambda, V_\Lambda,
M,\|\nabla F\|_{L^\infty(\mathcal{B}_M)}) \ \text{such that} \
diam\{\nabla u (B_\delta)\} \leq \eps$$ for any minimizer $u$ to
$I(u)$.

Since $O_\lambda, V_\Lambda$ are open sets covering
$\overline{\mathcal{B}}_M$, there exists $\ep^*>0$ such that any
ball $\mathcal{B}_{\ep^*}(p) \subset \mathcal{B}_{2M}$ is either
contained in $O_\lambda$ or in $V_\Lambda$.

\noindent Now, let $\ep < \ep^*$ and let $P_N$ be an $N$-th
regular polygon with sides of length $\ep/2,$ such that
$$\mathcal{B}_M \subset P_N \subset \mathcal{B}_{2M}.$$ Thus $N=N(\ep, M).$
Denote by $p_i$ the vertices of $P_N$, and let $T_{i}$ be the
triangle with vertices $p_i,p_{i+1}$ and $p_{i+2}, i=1,\ldots N
(p_{N+1}=p_1, p_{N+2}=p_2).$ Since the length of the sides of
$P_N$ is $\ep/2$, clearly $T_i$ is included in a ball
$\mathcal{B}_{\ep^*}(p) \subset \mathcal{B}_{2M}.$ Thus,
$$\text{either} \ T_i \subset O_\lambda \ \text{or} \ T_i\subset V_\Lambda.$$
Without loss of generality we can assume that for some
$i=1,\ldots, N,$ \begin{equation*}\label{nonempty}T_i \cap \nabla
u (B_1) \neq \emptyset,\end{equation*} otherwise we can work with
the polygon $P_N \setminus \bigcup_{i} T_{i}.$ Let $[p,q]$ be the
closed segment joining two points $p$ and $q$. Denote by $m_i$ the
middle point of the segments $[p_i,p_{i+1}]$. Also let $\e_i$ be
the direction such that
$$ [p_{i}, p_{i+2}]\subset \{p\cdot \e_i = l_i\}, \quad [m_i, m_{i+1}]\subset\{p \cdot
\e_i = l_i+\gamma\},$$ for some  $l_i,\gamma,$ with $0 < \gamma
<\eps/2.$ Then according to either Lemma \ref{A1} or Lemma
\ref{A2} we have that either
\begin{equation}\label{1}
\nabla u (B_\delta) \subset H^-(\e_i, l_i+\gamma),\end{equation}
or
\begin{equation}\label{2} \nabla u (B_\delta) \subset H^+(\e_i, l_i),\end{equation} for some
$\delta =\delta (\ep,\lambda,\Lambda, \omega_F, M,\|\nabla
F\|_{L^\infty(\mathcal{B}_M)}).$

\noindent If the latter holds then immediately $diam\{\nabla
u(B_\delta)\} \leq \ep$ as desired.

Assume that \eqref{2} does not hold for any choice of $i$. Then,
according to \eqref{1} we have that
$$\nabla u (B_{\delta^N}) \subset \widetilde{P}_N$$ where
$\widetilde{P}_N$ is the polygon with vertices the middle points
$m_i, i=1,\ldots, N.$ Now, we repeat this argument with
$\widetilde{P}_N$ instead of $P_N$ and iterate this a finite
number of times till we obtain a polygon of diameter smaller than
$\ep.$

 \qed

 \section{The proof of Theorem \ref{Main2}}

As in the previous section, let $F : \R^2 \rightarrow \R$ be a
smooth strictly convex function, and denote by $O_{1/n}$ an
increasing sequence of open sets such that
\begin{align*}
O_{1/n}& \subset \{p \in \R^2, D^2F(p) > \frac{1}{n}I\}.\\
\end{align*}
Recall that,
\begin{equation*}
I(u)= \int_{B_1} F(\nabla u)dx.
\end{equation*}

In this section we prove the following a priori estimate.

\

\noindent \textbf{Theorem 1.2.} \textit{Let $u$ be a minimizer to
$I(u)$ with $\|\nabla u\|_{L^\infty(B_1)} \leq M.$ Assume that
$\overline{\mathcal{B}}_M \setminus \bigcup O_{1/n}$ is a finite
set. Then in $B_{1/2}$, $\nabla u$ has a uniform modulus of
continuity depending on the sets $O_{1/n}$, $M$ and $ \|\nabla
F\|_{L^{\infty}(\mathcal{B}_{M})}.$}

\

The proof of Theorem 1.2 follows the same strategy as the proof of
 Theorem 1.1. Precisely, we determine a
 localization Lemma which allows us to apply an iteration argument
 and obtain the desired modulus of continuity.

 Here is the statement of the localization Lemma.

\begin{lem}\label{On}Let $u$ be a minimizer to $I(u)$ with
$\|\nabla u\|_{L^\infty(B_1)} \leq M.$ Assume that
\begin{equation}\label{hypOn} \nabla u (B_1) \cap
\mathcal{B}_\rho(p_0)=\emptyset, \quad \mathcal{B}_{4\rho}(p_0)
\cap \overline{\mathcal{B}}_M \subset O_{1/n} .\end{equation}Then,
there exists $\delta>0$ depending on $n,M, \rho, \|\nabla
F\|_{L^{\infty}(\mathcal{B}_{M})},$ such that either
\begin{equation}\label{ballinclusion}
\nabla u(B_\delta) \subset \mathcal{B}_{4\rho}(p_0)
\end{equation}or
\begin{equation}\label{noninter}
\nabla u(B_\delta) \cap \mathcal{B}_{3\rho}(p_0) = \emptyset.
\end{equation}
\end{lem}

In order to prove Lemma \ref {On}, we use the following
preliminary Lemma.

\begin{lem}\label{Onprel}Let $u$ be a minimizer to $I(u)$ in $B_2$ such
that
\begin{equation}\label{hypOnprel} \nabla u (B_1) \cap
\mathcal{B}_\eps(p_0)=\emptyset, \quad \nabla u (\partial B_1)
\cap \mathcal{B}_{\delta}(p_0) = \emptyset,\end{equation} with $0<
\eps< \delta.$ Then,
\begin{equation*} \nabla u(B_1) \cap \mathcal{B}_\delta(p_0) = \emptyset.
\end{equation*}
\end{lem}

\begin{proof}For simplicity, we take $p_0=0$.

In formula \eqref{weak}, let us choose $\textit{\textbf{e}}= e_k$
and as test function $\phi=\eta_k(\nabla u)$, with $\eta$ being 0
outside the ball $\mathcal{B}_\delta$.  Notice that $\phi$ is an
admissible test function in view of our hypothesis. We obtain,
\begin{equation*}
\int_{B_1} F_{ij}(\nabla u)u_{k j}(\eta_k(\nabla u))_i dx = 0.
\end{equation*} Hence differentiating and then summing over all
$k$'s, we get
\begin{equation}\label{inttrace}
\int_{B_1} Tr(D^2F D^2u D^2\eta D^2u) dx = 0.
\end{equation}

To estimate the integrand above, we use the characteristic
polynomial equation for a $2 \times 2$ matrix,
\begin{equation*}\label{char} A^2 -(TrA)A + (detA)I=0.
\end{equation*}
In particular, \begin{equation}\label{char2}TrA=0 \Rightarrow
A^2=-(det A) I.\end{equation}Thus, since $u$ solves the
Euler-Lagrange equation \eqref{EL}, we can apply this identity to
$A=D^2FD^2u$ to obtain,
\begin{align}\label{trace}Tr(D^2F D^2u D^2\eta D^2u)&= Tr((D^2F D^2u)^2D^2\eta
(D^2F)^{-1})\\ \nonumber & = det D^2F |det D^2u|
Tr((D^2F)^{-1}D^2\eta)
\end{align}
where in the first equality we used that
$$Tr(AB)=Tr(BA).$$
We choose $\eta$ so that
$$Tr((D^2F)^{-1}D^2\eta) >0 \quad \text{in $\mathcal{B}_\delta \setminus
\overline{\mathcal{B}}_{\eps/2}$}.$$ For example, for $\delta=1$
$$\eta(p)=
  \begin{cases}
    e^{-k|p|} - k e^{-k}|p| +(k-1)e^{-k}& \text{$|p| \leq 1$}, \\
    0 & \text{$|p|>1$},
  \end{cases}
$$ with $k$ a large enough constant depending on the ellipticity
constants for $D^2F$ on $\mathcal{B}_\delta.$ Combining
\eqref{inttrace} with \eqref{trace}, we get that
\begin{equation}
det D^2u = 0, \quad \text{on} \ (\nabla u)^{-1}(\mathcal{B}_\delta
\setminus \overline{\mathcal{B}}_{\eps/2})\cap B_1.
\end{equation}
Since $u$ solves the Euler-Lagrange equation we obtain that
$D^2u=0$ and $\nabla u$ is constant on the set above. Hence this
set is both open and closed in $B_1$, therefore it is either empty
(and we are done) or it coincides with $B_1$. In the latter case,
we conclude by continuity that also $\nabla u(\p B_1) \cap
\mathcal{B}_\delta \neq \emptyset$ and we reach a contradiction.

\end{proof}

Next we obtain a Caccioppoli-type inequality.

\begin{prop}\label{Cacc3}Let $u$ be a minimizer to $I(u)$ with
$\|\nabla u\|_{L^\infty(B_1)} \leq M.$ Then,
\begin{equation}\label{Caccineq3}
\int_{B_{1/2}\cap (\nabla u)^{-1}(O_{1/n})} |D^2u|^2 dx \leq C
\end{equation}
for some constant $C$ depending on $n, M, \|\nabla
F\|_{L^{\infty}(\mathcal{B}_{M})}$.
\end{prop}
\begin{proof} In formula \eqref{weak}, let us choose $\textit{\textbf{e}}= e_k$ and as test function $\phi= \xi^2
F_k(\nabla u)$, with $\xi \in C^\infty_0(B_1), 0 \leq \xi \leq 1$
and $\xi \equiv 1$ on $B_{1/2}$. We obtain, after summing over all
$k$'s,
\begin{equation}\label{Step1}
\int_{B_1} F_{ij}(\nabla u)u_{k j}\xi^2(F_k(\nabla u))_i dx = -2
\int_{B_1} F_{ij}(\nabla u) u_{k j} \xi \xi_i F_k(\nabla u) dx.
\end{equation} Now we analyze the left-hand side of
\eqref{Step1}. We use again \eqref{char2} for $A=D^2FD^2u$ and we
obtain
\begin{align*}
LHS&= \int_{B_1} F_{ij}(\nabla u)u_{k j}\xi^2 F_{kl}(\nabla
u)u_{li} dx = \int_{B_1} \xi^2 Tr((D^2F D^2u)^2)dx\\& = 2
\int_{B_1} \xi^2 (detD^2F)|det D^2u| dx.
\end{align*}
Moreover, again using \eqref{char2} we have that
$$|D^2u|^2\leq |(D^2F)^{-1}|^2|D^2FD^2u|^2 = |(D^2F)^{-1}|^2|det(D^2FD^2u)|$$
and hence,
$$det D^2F |detD^2u|\geq c|D^2u|^2 \quad \text{on} \ (\nabla
u)^{-1}(O_{1/n}),$$ for some constant $c$ depending on $n$. Thus,
\begin{equation}\label{finalbound} LHS \geq c \int_{B_{1/2}\cap (\nabla u)^{-1}(O_{1/n})} |D^2 u|^2 dx\end{equation}
with $c$ depending on $n$.

On the other hand, using integration by parts together with the
Euler-Lagrange equation $\p_k (F_k(\nabla u))=0$, we have that the
right-hand side of \eqref{Step1} is given by
\begin{align*}
RHS  & = - 2  \int_{B_1} \partial_k(F_j(\nabla u))\xi\xi_i
F_k(\nabla u)dx
\\& = 2 \int_{B_1} F_j(\nabla u)
\partial_k (\xi \xi_i) F_k(\nabla u)dx, \end{align*} and hence
\begin{equation} \label{Bound} |RHS| \leq C
\end{equation}
with $C$ depending on $\|\nabla F\|_{L^\infty(\mathcal{B}_M)}.$
Combining \eqref{Step1},\eqref{finalbound} and \eqref{Bound} we
obtain the desired claim.

\end{proof}

Finally, we are ready to exhibit the proof of the localization
Lemma \ref{On}.

\

\noindent \textbf{Proof of Lemma \ref{On}.} Let $\delta>0$ and
assume that there exist $x_0,x_1 \in B_\delta$ such that
$$|\nabla u (x_0)-p_0|\geq 4\rho , \quad |\nabla u (x_1)- p_0| <
3\rho .$$ Then, by the maximum principle for each $\delta \leq r
\leq 1/2$ there exists a point $x_0^r \in
\partial B_r$ such that $$|\nabla u (x_0^r)-p_0|\geq 4\rho.$$
Also, from Lemma \ref{Onprel}, for each $\delta \leq r \leq 1/2$
there exists a point $ x_1^r \in
\partial B_r$ such that $$|\nabla u
(x_1^r)- p_0| < 3\rho.$$

Let $\mathcal{G}$ be a smooth function on $\R^2$ with $|\nabla
\mathcal{G}| \leq 1/\rho$ such that $$\mathcal{G}(p)=
  \begin{cases}
    0 & \text{$|p-p_0|\geq 4\rho$}, \\
    1 & \text{$|p-p_0|\leq 3\rho$}.
  \end{cases}
$$
We proceed as in Lemma \ref{A2} to obtain,
\begin{equation}\label{Semifinal}
\int_{B_{1/2}} |\nabla (\mathcal{G}(\nabla u))|^2dx \geq
\frac{1}{2\pi}\ln \frac{1}{2\delta}.
\end{equation}
However, since $\mathcal{B}_{4\rho}(p_0) \cap
\overline{\mathcal{B}}_M \subset O_{1/n}$, by the definition of
$\mathcal{G}$ we get
\begin{align}\label{Final}
\int_{B_{1/2}} |\nabla (\mathcal{G}(\nabla u))|^2dx &\leq
\int_{B_{1/2}} |\nabla \mathcal{G}|^2 |D^2u|^2dx \leq
\frac{1}{\rho^2}\int_{B_{1/2} \cap (\nabla u)^{-1}(O_{1/n})} |D^2
u|^2dx.
\end{align}
Therefore, combining \eqref{Semifinal} with \eqref{Final}, we
obtain a contradiction to the Caccioppoli inequality
\eqref{Caccineq3} as long as $\delta$ is sufficiently small.

\qed

\

We now combine all the ingredients above and provide the proof of
Theorem 1.2.

\

\noindent\textbf{Proof of Theorem 1.2.} We wish to prove the
following claim:
$$\forall \eps >0, \exists \delta>0, \delta=\delta(\ep, O_{1/n},
M,\|\nabla F\|_{L^\infty(\mathcal{B}_M)}) \ \text{such that} \
diam\{\nabla u (B_\delta)\} \leq \eps$$ for any minimizer $u$ to
$I(u)$.

\

Fix $\eps>0$. Let $\overline{\mathcal{B}}_M \setminus \bigcup
O_{1/n} = \{q_1,\ldots,q_m\}$. We cover the set $\mathcal{B}_{2M}
\setminus \bigcup_{i=1}^m \mathcal{B}_{5\rho}(q_i)$ with a finite
number of balls $\mathcal{B}^k_\rho$ of radius $\rho$, with $\rho$
small, say $\rho=\eps/5$. Notice that $\mathcal{B}^k_{4\rho} \cap
\overline{\mathcal{B}}_M \subset O_{1/n}$, for a large $n$.

Clearly, since $\nabla u (B_1) \subset \overline{\mathcal{B}}_M$,
then $\nabla u (B_1) \cap \mathcal{B}^k_\rho=\emptyset$  for some
$\mathcal{B}^k_\rho \subset \mathcal{B}_{2M} \setminus
\overline{\mathcal{B}}_M$. Then, according to Lemma \ref{On} we
have that either
\begin{equation}\label{ballinclusion2}
\nabla u(B_\delta) \subset \mathcal{B}^k_{4\rho}
\end{equation}or
\begin{equation}\label{noninter2}
\nabla u(B_\delta) \cap \mathcal{B}^k_{3\rho} = \emptyset,
\end{equation}
for some $\delta>0.$ If \eqref{ballinclusion2} occurs, then we
reached our conclusion. Otherwise, if \eqref{noninter2} occurs, we
conclude that $\nabla u (B_\delta) \cap \mathcal{B}^j_\rho
=\emptyset$ for all $\mathcal{B}^j_\rho$ such that $
\mathcal{B}^j_\rho \cap \mathcal{B}^k_{\rho} \neq \emptyset.$
Hence we can apply again Lemma \ref{On} to the balls
$\mathcal{B}^j_\rho.$ We iterate this argument. If at some step we
reach the conclusion \eqref{ballinclusion2}, then we are done. If
at each step we reach the conclusion \eqref{noninter2}, then after
a finite number of steps (because our balls cover a connected
domain) we obtain that for some small $\tilde{\delta}$, $\nabla
u(B_{\tilde{\delta}}) \cap \mathcal{B}_\rho^k = \emptyset$ for all
$k$. Since $\nabla u(B_{\tilde{\delta}})$ is connected, we
conclude that $\nabla u(B_{\tilde{\delta}}) \subset
\mathcal{B}_{5\rho}(q_i)$ for some $i$. Again we reach the desired
conclusion.

\qed

\section{A degenerate obstacle problem}

\subsection{The statement of the problem and preliminaries.} Let $N$ be a convex (open) polygon in $\mathbb{R}^2$
with $n$ vertices $\{p_1,\ldots, p_n\} = \mathcal{P}$ (also, set
$p_{n+1}=p_1.$) Let $[p,q]$ be the closed segment joining two
points $p$ and $q$ and let $(p,q)$ be the open segment joining
them. Denote by $\nu_i$ a direction perpendicular to the side
$[p_i,p_{i+1}], i=1,\ldots, n$. Finally, let
$\mathcal{Q}=\{q_1,\ldots, q_m\}$ be a finite subset of $N$.

Let $F: \overline{N} \rightarrow \mathbb{R}$ be a convex function
such that \begin{enumerate}
\item $F \in C^2(N \setminus \mathcal{Q}), \quad D^2F >0$ on $N \setminus \mathcal{Q}$;
\item  $F$ is bounded.
\end{enumerate}

Let $\Omega \subset \mathbb{R}^2$ be a bounded Lipschitz domain,
and let $\varphi: \partial \Omega  \rightarrow \mathbb{R}$ be a
function that admits an extension $\tilde{\varphi}$ with $\nabla
\tilde{\varphi} \in \overline{N}.$

We consider the following problem $(P)$: minimize the functional
\begin{equation*} I(u)= \int_{\Omega} F(\nabla u)dx,
\end{equation*}among all Lipschitz competitors $u$, with
$u=\varphi$ on $\partial \Omega, \nabla u \in \overline{N}$ (we
think of $F$ to be equal to $+\infty$ outside $\overline{N}$.)

Denote by,
\begin{align*}
\underline{\varphi}= \inf \{v: v=\varphi \ \text{on} \ \partial
\Omega, \nabla v \in \overline{N}\};\\
\ \\\overline{\varphi}= \sup \{v: v=\varphi \ \text{on} \ \partial
\Omega, \nabla v \in \overline{N}\}.
\end{align*}
We refer to $\underline{\varphi},\overline{\varphi}$ as to
respectively, the lower obstacle and upper obstacle.

Also, in what follows we want to distinguish when gradients are
close or not to $\partial N$. One way of doing this is to consider
the compactification of $N$ with one point. We adopt a slightly
different approach by introducing the following function
\begin{equation}\label{H}
H:\overline{N} \to S^2\end{equation} such that, $H$ is continuous,
$\partial N$ is mapped to a point, and $H$ is a homeomorphism
between $N$ and $H(N)$.

\

\textbf{Remark.} Notice that our minimization problem is
equivalent to the following obstacle problem: minimizing
\begin{equation*} \widetilde{I}(u)= \int_{\Omega} \widetilde{F}(\nabla
u)dx,
\end{equation*}
 among all
competitors $v$ such that $\underline{\varphi} \leq v \leq
\overline{\varphi}$, with $\widetilde{F}$ a convex extension of
$F$ outside $\overline{N}$.

\qed

\subsection{The main results.} We state now our main results. Let $u$ be the unique minimizer to the
problem $(P)$.

Our first main result says that the minimizers are $C^1$ in
$\Omega$ except on a number of segments which have an end point on
$\partial \Omega$ and have directions perpendicular to the sides
of $N$. On these segments the minimizer coincides with one of the
lower or upper obstacle.

\

\textbf{Theorem 1.3}. \textit{If $\nabla u$ is discontinuous at
$x_0 \in \Omega$, then there exists a direction $\nu_i$
perpendicular to one of the sides $[p_i,p_{i+1}]$ such that $u$ is
linear on a segment of direction $\nu_i$ connecting $x_0$ and
$\partial \Omega$. Precisely, there exists $x_1=x_0 + t_1 \nu_i
\in \partial \Omega$ such that
$$u(x) = u(x_0)+ p_i \cdot (x-x_0), \quad \text{for all} \ x \in [x_0,x_1] \subset \overline{\Omega}.$$
In particular, $u \in C^1$ away from the obstacles.}

\

Our second main result says that if a sequence of points converges
to a point of non-differentiability of $u$ then their
corresponding gradients approach $\partial N$. Let $H$ be as in
\eqref{H}; the precise statement of our result reads as follows.

\begin{thm}\label{continuous}$H(\nabla u)$ is continuous in $\Omega$.\end{thm}

Minimizers of $(P)$ may develop ``flat" regions where they are
linear and their gradient belongs to the set of vertices of $N$.
The next theorem describes the shape of these regions.

\begin{thm}\label{convex} Let $p_i$ be a vertex of $N$ and let $\omega_i$ be a direction
so that
$$\omega_i \cdot(p -p_i) >0 \quad \forall p\in \overline{N}
\setminus \{ p_i\}.$$ Let $S$  be the interior of the set $\{u(x)=
c + p_i\cdot x\}$ and assume that $S \neq \emptyset.$ Then,
$\partial S \cap \Omega$ consists of a convex graph (by above) and
a concave graph (by below)in the $\omega_i$ direction.
\end{thm}

Finally, the next two results deal with the case when $u$ is a
perturbation of a linear function. They will be used in the proof
of our main Theorem 1.3.

\begin{thm}\label{Flat1}Assume $B_1 \subset \Omega$ and $$\mathcal{B}_{\delta}(p_0) \subset
N.$$Then, there exists $\eps$ depending on $\delta, p_0, F,$ such
that if $$|u - p_0 \cdot x|\leq \eps, \quad x \in B_1$$ then $u
\in C^1(B_{1/2})$ and
$$\nabla u (B_{1/2}) \subset \mathcal{B}_{\delta}(p_0).
$$
\end{thm}

\begin{thm}\label{Flat2} Assume $B_1 \subset \Omega$ and let $p_0 \in \partial N.$ Then, there
exists $\eps$ depending on $\delta, F$, such that if
$$|u - p_0\cdot x| \leq \eps, \quad x\in B_1$$ then
$$[\nabla u(x), p_0] \quad \text{is in a $\delta$-neighborhood of $\partial N$, for a.e $x \in B_{1/2}.$}$$
\end{thm}

We finish this section by showing that minimizers to the problem
$(P)$ are unique. We use the abbreviation L.P. for Lebesgue point.

\begin{prop} The minimization problem $(P)$ admits a unique solution.
\end{prop}
\begin{proof}
Let $u_1, u_2$ be two distinct solutions. By the convexity of $F$,
we have that $(u_1+u_2)/2$ is also a solution and
$$\frac{F(\nabla u_1) +F(\nabla u_2)}{2} = F(\frac{\nabla u_1 + \nabla u_2}{2}), \quad \text{a.e. in $\Omega$.}$$
Hence, since $F$ is strictly convex in $N$, we have that if $x$ is
a L.P. for $\nabla u_1$ and $\nabla u_2$ then either
\begin{equation}\label{equal}\nabla u_1(x)= \nabla
u_2(x)\end{equation} or
\begin{equation}\label{sides} \nabla u_1(x),\nabla
u_2(x) \in [p_i,p_{i+1}], \quad \text{for some $i=1,\ldots, n.
$}\end{equation}

Now, let us assume by contradiction that $u_1(0)
> u_2(0),$ $0 \in \Omega,$ say $u_1(0) - u_2(0)=\eps.$ Then,
since $u_1,u_2$ are Lipschitz and they coincide on the boundary,
there exists $\rho=c \eps$ such that if $|\tau| < \rho$, then
$$B_{\rho} \subset \{u_2(x+\tau) +\eps/2 < u_1\} \subset
\subset \Omega.$$ The minimality of $u _2(x+\tau)$ together with
the inclusion above imply that $$v_\tau(x)=\min\{u_1(x),
u_2(x+\tau) + \eps/2\}$$ is also a minimizer of the problem $(P)$.

Thus, if $x_1\in B_{\rho/2}$ is a L.P. for $\nabla u_1$ and
$x_2\in B_{\rho/2}$ is a L.P. for $\nabla u_2$, we can translate
$x_2$ by $\tau$ to coincide with $x_1$, so that $x_1$ is a L.P.
for $v_\tau$ and according to \eqref{equal}-\eqref{sides} we have
either
\begin{equation}\label{equal2}\nabla
u_1(x_1)= \nabla u_2(x_2)\end{equation} or
\begin{equation}\label{sides2} \nabla u_1(x_1),\nabla
u_2(x_2) \in [p_i,p_{i+1}], \quad \text{for some $i=1,\ldots, n.
$}\end{equation}

Similarly, by appropriate translations we obtain that if $x_1 \in
B_{\rho/2}$ is a L.P. for $\nabla u_1$ and $x_2,x_3,x_4 \in
B_{\rho/2}$ are L.P. for $\nabla u_2$ then
\begin{align*}&F(\frac{1}{4}(\nabla u_1(x_1) + \nabla u_2(x_2) + \nabla
u_2(x_3)+\nabla u_2(x_4)))\\
\\ & = \frac{1}{4}(F(\nabla u_1(x_1)) + F(\nabla u_2(x_2)) + F(\nabla
u_2(x_3))+F(\nabla u_2(x_4))).\end{align*} Hence
\begin{equation}\label{collinear}\nabla u_1(x_1), \nabla u_2(x_2), \nabla
u_2(x_3),\nabla u_2(x_4) \ \text{are collinear.}\end{equation}

\

We distinguish two cases.

\

\textit{Case 1.} There exists at least one L.P. $\overline{x}\in
B_{\rho/2}$ for either $\nabla u_1$ or $\nabla u_2$ such that
$\overline{p}=\nabla u_i(\overline{x}) \in N$. Then, we conclude
that $\nabla u_1 = \nabla u_2= \overline{p}$ in $B_{\rho/2}$.
Hence $u_1-u_2=\eps$ on $B_{\rho/2}.$ Now we can proceed as above
with 0 replaced by any point in $B_{\rho/2}.$ By iterating this
argument a finite number of times we conclude that
$\{u_1-u_2=\eps\}$ must coincide with $\Omega$. This contradicts
that $u_1=u_2$ on $\partial \Omega.$

\

\textit{Case 2}. All Lebesgue points $x\in B_{\rho/2}$ for $\nabla
u_i, i=1,2$ are mapped on $\partial N$.

First we claim that all Lebesgue points $x\in B_{\rho/2}$ for
$\nabla u_i, i=1,2$ are mapped on the same side. Indeed, according
to \eqref{collinear} any three such points for $\nabla u_2$ are
collinear. The claim follows by interchanging $u_2$ and $u_1$.

Thus all Lebesgue points of $\nabla u_1, \nabla u_2$ in
$B_{\rho/2}$ must lie on the same side, say $[p_1,p_2]$. Therefore
$u_1-u_2$ is constant $\eps$ in $B_{\rho/2}$ on the segment with
middle point at zero, in the direction perpendicular to
$[p_1,p_2]$. This implies that the function $\nu \cdot x$ cannot
achieve a maximum (minimum) on the set $\{u_1-u_2 =\eps\}$, where
$\nu$ is a direction which differs from any of the perpendicular
directions to the sides $[p_i, p_{i+1}]$. This contradicts that
$\Omega$ is bounded.

\end{proof}


\section{The approximation.}

The goal of this section is to obtain smooth approximations $u_m$
that converge uniformly to the minimizer $u$ of the problem $(P)$.
Moreover, for any compact $K \Subset \Omega$ we want $\nabla
u_m(K)$ to lie in any neighborhood of $\overline{N}$ for $m$ large
enough. One way of achieving this is to approximate $F$ with
smooth convex functions $F_m$ that converge to $\infty$ outside
$\overline{N}$ and have cubic growth at $\infty.$

Let $\overline{F}$ be a convex function on $\R^2$, $\overline{F}=
0$ on $N$, $\overline{F}(p)= |p|^3$ for $|p|$ large, $\overline{F}
\in C^{\infty}(\R^2\setminus \overline{N})$ with $D^2\overline
F>0$ on $\R^2\setminus \overline{N}$.

Let $F_m \in C^{\infty}(\mathbb{R}^2)$ with $D^2F_m >0$, be such
that

\begin{enumerate}
\item $F_m \rightarrow F$ uniformly on $\overline{N}$;
\item $D^2 F_m \rightarrow D^2F$ uniformly on compacts on $N \setminus \mathcal{Q};$
\item $F_m(p) =C_m \overline{F}$ in $D_m:=\{p\in \R^2 |
\overline{F}(p)>1/m\},$ with $C_m \rightarrow +\infty.$
\end{enumerate}

Let $u_m$ be the minimizer to $$\int_\Omega F_m(\nabla v)dx, \quad
\quad v=\varphi \quad \mbox{on $\p \Omega$}. $$ Here $\varphi$ is
the boundary data of the minimizer of the problem $(P).$

We will show that the $u_m$'s are the desired smooth
approximations. Using the lower obstacle $\underline{\varphi}$ as
competitor, we see that $\int_\Omega F_m(\nabla u_m)dx$ is bounded
above by a fixed constant. Hence, since $C_m >1$ we conclude
$$\int_\Omega |\nabla u_m|^3 dx  \leq C.$$ By Sobolev embedding theorem
we have that $u_m$ is uniformly H\"older continuous in
$\overline{\Omega}.$ Thus, by Ascoli-Arzela there exists a
function $u$ such that (for a subsequence of $m$'s)
$$u_m \to u \quad \mbox{uniformly on $\overline{\Omega}$}.$$  First we show the following Proposition (recall the definition
\eqref{H} of $H$ from the previous section, and assume $H$ to be
defined on the whole $\R^2$ and to be constant outside $N$).

\begin{prop}\label{convergence} $u$ is the minimizer to the problem $(P)$. Also $$\nabla u_m \rightarrow \nabla
u \ \text{in measure} \ \text{on} \ A:=\{x | x \ \text{is a L.P.
for $\nabla u$ and $\nabla u(x) \in N$}\},$$ and $$H(\nabla u_m)
\rightarrow H(\nabla u) \ \text{in measure on} \ \Omega.$$
\end{prop}

\begin{proof}
Let $v$ be a function on $\overline{\Omega}$, $v = \varphi$ on
$\partial \Omega$ and $\nabla v \in \overline{N}.$ Let
$\widetilde{F}$ be a convex function on $\mathbb{R}^2$ with
bounded gradient which approximates $F$ in $N$. Then for $m$
large,

\begin{align*}
\int_\Omega F(\nabla v)dx &\geq \int_\Omega F_m(\nabla v)dx -\eps
\geq \int_\Omega  F_m(\nabla u_m)dx -\eps \geq \int_{\Omega}
\widetilde{F}(\nabla u_m)dx -2\eps\\ & \geq \int_{\Omega}
(\widetilde{F}(\nabla u) + \nabla \widetilde{F}(\nabla u)(\nabla
u_m - \nabla u) + \omega_{\widetilde{F}}(\nabla u, \nabla u_m -
\nabla u))dx -2\eps,
\end{align*} with $\omega_{\widetilde{F}} \geq 0$ the modulus of convexity of
$\widetilde{F}$. Using that $\nabla u_m \rightarrow \nabla u$
weakly in $L^3$, we conclude that
$$\int_\Omega F(\nabla v)dx \geq \int_\Omega \widetilde{F}(\nabla
u)dx -2\eps.$$ Since $\widetilde F$ is arbitrary outside
$\overline N$, we deduce that $\nabla u \in \overline{N}$ and $u$
is the minimizer of the problem $(P)$.

Since $F$ is strictly convex in $N$ and
$$\int_\Omega \omega_{\widetilde{F}}(\nabla u, \nabla
u_m - \nabla u)dx \rightarrow 0, \quad m\rightarrow \infty$$ we
also obtain that $\nabla u_m$ converges in measure to $\nabla u$
in $A$, and $H(\nabla u_m)$ converges in measure to $H(\nabla u).$

\end{proof}

We continue with the following Proposition.

\begin{prop}\label{compact}$u_m \in C^\infty(\Omega)$. Also, for any compact $K \Subset \Omega$
and $\delta>0$, $\nabla u_m(K)$ is in a $\delta$-neighborhood of
$\overline{N}$ for $m$ large enough.
\end{prop}
\begin{proof}
The proof is standard and follows the lines of interior Lipschitz
estimates for $p$-harmonic functions.

Using the lower obstacle $\underline{\varphi}$ as competitor, we
see that $\int_\Omega F_m(\nabla u_m)dx$ is bounded above by a
fixed constant. Hence,
\begin{equation}\label{limit} \int_\Omega \overline{F}(\nabla
u_m)dx \rightarrow 0, \ \text{as $m \rightarrow +\infty$}.
\end{equation}

For notational simplicity we denote $u_m = \tilde{u}$ and
$F_m=\widetilde F$.

Let $\eta$ be a convex function, $\eta =0$ outside $D_{m_0} =
\{\overline{F}> 1/m_0\}$ for some fixed $m_0$, $\eta(p)=|p|$ for
large $|p|$. Also, let
$$\psi= \eta(\nabla \tilde{u}) +\frac{1}{2}\eta^2(\nabla \tilde{u}).$$ Then $\psi$ is a subsolution to the
following elliptic equation
\begin{equation}{\label{5.2}} \p_i (a_{ij}\psi_j)\geq 0, \quad a_{ij}=\frac{\overline{F}_{ij}(\nabla \tilde{u})}{1+\eta(\nabla \tilde{u})} .
\end{equation}
Indeed, since $\widetilde{F}=C_m\overline{F}$ in $D_{m_0}$ (we
assume $m>m_0$) and $\eta, \psi =0$ when $\nabla \tilde u$ is
outside $D_{m_0}$, we can replace $\overline{F}$ by
$\widetilde{F}$ in our computations. Using the Euler-Lagrange
equation for $\tilde u$
$$\p_i (\widetilde{F}_{ij}\tilde{u}_{kj})=0,$$ it is
straightforward to check that
$$\p_i (a_{ij}\psi_j)= \widetilde{F}_{ij}\eta_{kl}\tilde{u}_{li}\tilde{u}_{kj} \geq 0.$$

Notice that the equation \eqref{5.2} is uniformly elliptic since
in $D_{m_0}$ we have
$$\lambda_{m_0}(1+\eta) I \leq D^2\overline{F} \leq \Lambda_{m_0}(1+\eta)I.$$
Now we can apply the standard estimate (see \cite{GT}, Theorem
8.17) and obtain that for any compact $K \Subset \Omega,$
$$\|\psi\|_{L^\infty(K)} \leq C
\|\psi\|_{L^\alpha(\Omega)}, \quad \mbox{for $\alpha >1$,}$$ with
$C$ depending on $m_0$, $K$, $\alpha$. We choose $\alpha=3/2$ and
use
$$\psi^{3/2} \le C(1+|\nabla \tilde u|^3) \le C(m_0) \overline F(\nabla \tilde u)$$ together with \eqref{limit} to obtain
$$\|\psi\|_{L^\infty(K)} \to 0 \quad \mbox{as $m \to \infty$},$$
which implies the second statement of our proposition.

Since $ \nabla \tilde u$ is locally bounded the first part
($\tilde u \in C^{\infty}$) follows from the classical theory.

\end{proof}

\textbf{Remark.} The smoothness of $u_m$ follows also from the
fact that $u_m$ solves a uniformly elliptic equation in 2D,
therefore $u_m$ is $C^{1,\alpha}$ in the interior and hence it is
$C^\infty.$

\qed

Our analysis will rely on the following classical theorem that was
also used in other two dimensional results (see for example
\cite{GT}, \cite{H}, \cite{S}).

\begin{thm}\label{components}Let $v$ be a solution to
$$a_{ij}(x)v_{ij}(x) = 0 \quad \text{in $D \subset \R^2$,}
$$ with $A(x)={a_{ij}(x)} >0, A(x)\in C^{\infty},$ and $D$ simply
connected. Assume $v$ is not linear. Then in each neighborhood $U$
there exists a point $x_U$ such that each set $$\{v > l_{U}\},
\quad \{v < l_{U}\}$$ with
$$l_{U}(x) :=v (x_U) + \nabla v(x_U)\cdot (x-x_U)$$ has at least
two connected components in $D$ that intersect $U$. Moreover these
components are not compactly supported in $D$. \end{thm}

\begin{proof}Let $x_U \in U $ be such that $det D^2v(x_U) <0$.
Such a point exists otherwise if $\det D^2v =0$ in $U$, then $D^2v
=0$ in $U$ and by unique continuation $v$ is linear in $D$.

Clearly, the sets $\{v > l_{U}\}, \{v < l_{U}\}$ intersect a small
ball around $x_U$ precisely in four  disjoint connected
components.

On the other hand, it follows from the maximum principle that any
connected component of the sets $\{v > l_{U}\}, \{v < l_{U}\}$
cannot be compactly supported in $D$. Hence, since $D \subset
\R^2$ and $D$ is simply connected, the components in the small
ball belong to four disjoint connected components in $D.$

\end{proof}

\section{The proof of Theorem 1.3.}

In this section we exhibit the proof of Theorem 1.3. We start by
obtaining a result to which we refer to as the localization
Theorem. From now on we tacitly assume that our statements hold
for all $m$ sufficiently large.

\begin{thm}\label{localization} Assume $B_1 \subset \Omega$ and
\begin{equation}\label{hardassumption}\nabla u_m (B_1) \cap \mathcal{B}_{\rho}(p_0)=\emptyset,
\quad p_0 \in \overline{N}.\end{equation} Then for any $\eps>0$,
there exists $\delta = \delta(\eps, F, \mathcal{B}_{\rho}(p_0))$
such that either
$$\nabla u_m (B_\delta) \subset \mathcal{B}_\eps(p), \ \text{for some $p \in N$},$$ or $$\nabla u_m (B_\delta) \subset
\mathcal{N}_\eps$$ with $\mathcal{N}_\eps$ the $\eps$ neighborhood
of $\partial N.$
\end{thm}

\textbf{Remark.} Another way of stating the conclusion of this
theorem is to say that $H(\nabla u_m)$ is continuous with a
uniform modulus of continuity.

Also recall that $H(\nabla u_m) \rightarrow H(\nabla u)$ in
measure, thus if the hypothesis of Theorem \ref{localization}
holds for all large $m$'s, then $u$ satisfies the same conclusion
of the Theorem.

\qed

\begin{proof}
Assume $$\nabla u_m (B_1) \cap \mathcal{B}_{r}(p)=\emptyset, \quad
\overline{\mathcal{B}_{4r}(p)} \subset N \setminus \mathcal{Q}.$$

We wish to prove that there exists $\delta>0$ depending on $F,
\mathcal{B}_r(p),$ such that either
\begin{equation}\label{1ballinclusionm}
\nabla u_m(B_\delta) \subset \mathcal{B}_{4r}(p),
\end{equation}or
\begin{equation}\label{1noninterm}
\nabla u_m(B_\delta) \cap \mathcal{B}_{3r}(p) = \emptyset.
\end{equation}

We argue similarly as for Lemma 3.1. It suffices to prove the
following Caccioppoli-type inequality,
\begin{equation}\label{new1}
\int_{B_{1/2} \cap (\nabla u_m)^{-1}(\mathcal{B}_{4r}(p))} |D^2
u_m|^2dx \leq C,
\end{equation}
for some constant $C$ depending on $F, \mathcal{B}_{r}(p).$

For notational simplicity let $F_m= \widetilde{F}$ and $u_m =
\tilde{u}$. We have

\begin{equation*}\label{weak2}
\int_{B_1} \widetilde{F}_{ij} (\nabla \tilde{u}) \tilde{u}_{k j}
\phi_i dx =0, \ \forall \phi \in C_0^\infty(B_1).
\end{equation*}

Let us choose $\phi= \xi^2 \eta_k(\nabla \tilde{u})$, with $\eta$
compactly supported on $\overline{\mathcal{B}_{4r}(p)}$ and $\xi
\in C^\infty_0(B_1), 0 \leq \xi \leq 1, \xi \equiv 1$ on
$B_{1/2}$. We obtain, after summing over all $k$'s (when clear, we
drop the dependence on $\nabla \tilde{u}$),
\begin{equation}\label{Step2}
\int_{B_1} \widetilde{F}_{ij}\tilde{u}_{k j}\xi^2(\eta_k(\nabla
\tilde{u} ))_i dx = -2 \int_{B_1} \widetilde{F}_{ij} u_{k j} \xi
\xi_i \eta_k dx.
\end{equation} Now we analyze the left-hand side of
\eqref{Step2}. We proceed as in Lemma \ref{Onprel} and we obtain
\begin{align*}
LHS&= \int_{B_1} \widetilde{F}_{ij}\tilde{u}_{k j}\xi^2
\eta_{kl}\tilde{u}_{li} dx \\& = \int_{B_1} \xi^2
(detD^2\widetilde{F})|det D^2\tilde{u}|
Tr((D^2\widetilde{F})^{-1}D^2\eta)dx.
\end{align*}
Choose $\eta$ such that $$Tr((D^2\widetilde{F})^{-1}D^2\eta) \geq
c, \quad \text{on $ \mathcal{B}_{4r}(p) \setminus
\mathcal{B}_r(p)$},
$$ for a small constant $c$ depending on the
ellipticity constants of $D^2F$ on $\mathcal{B}_{4r}(p)$, say
$\lambda, \Lambda$. We conclude
\begin{equation}\label{finalbound2} LHS \geq c \int_{B_{1 }\cap (\nabla \tilde{u})^{-1}(\mathcal{B}_{4r}(p))}
|D^2 \tilde{u}|^2 \xi^2 dx\end{equation} with $c$ depending on
$\lambda,\Lambda$.

On the other hand, the right-hand side of \eqref{Step2} is bounded
by
\begin{align}\label{Bound2}
|RHS|  & \leq  \gamma C \int_{B_1 \cap (\nabla
\tilde{u})^{-1}(B_{4r}(p))}|D^2\tilde{u}|^2\xi^2 dx + \tilde{C}
\end{align}
with $C$ depending on $\Lambda$. Combining
\eqref{Step2},\eqref{finalbound2} and \eqref{Bound2} we obtain the
desired inequality.

We now proceed similarly as in the proof of Theorem 1.2. We cover
the set $N \setminus (\mathcal{N}_{\eps} \cup \mathcal{B}_{\eps
}(q_i))$ as a finite union of balls $\mathcal{B}^k_r$ of radius
$r=\eps /5$ (with centers in the set). From our assumption, there
exists $k$ such that
$$\nabla u_m (B_1) \cap \mathcal{B}^k_{r} =\emptyset,$$ provided
that $\eps$ is small enough depending on $\rho, F$. Now the
conclusion follows from the same iteration argument as in Theorem
1.2 (notice that by Proposition \ref{compact}, $\nabla u_m
(B_\delta)$ is in an $\eps$ neighborhood of $\overline{N}$ for
large $m$.)

\end{proof}

\subsection{The proof of Theorem 1.3.} We now present a series of Propositions which will all
be combined towards the proof of Theorem 1.3. We start by stating
two Propositions for the approximation $u_m$ which correspond to
the flatness Theorems \ref{Flat1}-\ref{Flat2} for the minimizer
$u$. We present their proofs in the last section.

\begin{prop}\label{flat1}Assume $B_1 \subset \Omega$ and $$\mathcal{B}_{\delta}(p_0) \subset
N.$$Then, there exists $\eps$ depending on $\delta, p_0, F,$ such
that if $$|u_m - p_0 \cdot x|\leq \eps, \quad x \in B_1$$ then
$$\nabla u_m (B_{1/2}) \subset \mathcal{B}_{\delta}(p_0).
$$
\end{prop}

\begin{prop}\label{flat2} Let $p_0 \in \partial N.$ Then, there
exists $\eps$ depending on $\delta, F$, such that if
$$|u_m - p_0\cdot x| \leq \eps, \quad x \in B_1$$ then
$$[\nabla u_m(x), p_0] \quad \text{is in a $\delta$-neighborhood of $\partial N$, for all $x \in B_{1/2}.$}$$
\end{prop}

In order to apply the localization Theorem near the origin (see
\eqref{hardassumption}), we need to find a ball $\mathcal{B}_r(p),
p\in \overline{N},$ that does not intersect the image $\nabla u_m
(B_r)$ (with $r$ small and $\mathcal{B}_r(p)$ depending also on
$u$). This is not always possible.

The next Proposition, which is key in proving Theorem 1.3, states
a condition which guarantees the existence of such a ball in a
neighborhood of a side of $N$. Its proof relies on the previous
flatness results.

\begin{prop}\label{segments}Assume $0 \in (p_1,p_2)$ and $e_2$ is
normal to $(p_1,p_2)$ and points inside $N$. Let $0\in [a_1,a_2]
\subset (p_1,p_2)$ and assume that for each $r>0$ and $p\in
[a_1,a_2],$ there exists a sequence of $m \rightarrow \infty$ such
that
$$\nabla u_m(B_r) \cap \mathcal{B}_r(p)
\neq \emptyset.
$$
Then $u$ is constant on a segment of direction $e_2$ connecting
$0$ and $\partial \Omega$. Precisely, there exists $\tilde{x}= s
e_2 \in
\partial \Omega$ such that
$$u(x) = u(0), \quad \text{for all} \ x \in [0,\tilde{x}] \subset \overline{\Omega}.$$
\end{prop}

In order to prove Proposition \ref{segments} we will need the
following Lemma.

\begin{lem}\label{R} Assume $0 \in \partial N,$ $e_2$ points inside $N$,
$$R:= \{|x_1|< \delta, |x_2| < 1\} \Subset \Omega,$$ and
$$u(x)> 0, \ x \in \overline{R} \cap \{x_2=1\}, \quad u(x)< 0, \ x \in \overline{R} \cap \{x_2=-1\}, \quad
u(0)=0.$$ Assume that for each $r>0,$ there exists a sequence of
$m \rightarrow \infty$ such that
$$\nabla u_m(B_r) \cap \mathcal{B}_r(0)
\neq \emptyset. $$ Then either the set $$\{u=0\}\cap \{x_1 > 0
\}$$ or the set $$\{u=0\}\cap \{x_1 < 0 \}$$
 is given by the region between a convex graph (from
above) and a concave graph (from below) in the $e_2$ direction.
\end{lem}
\begin{proof}The lemma holds trivially if $u$ is linear.
Assume $u$ is not linear. Then, for $m$ large enough, also $u_m$
is not linear. Then, by the assumptions together with Theorem
\ref{components}, we have that there exists $x_m \rightarrow 0$
with $\nabla u_m (x_m) \rightarrow 0$ such that the set $$\{u_m <
l_m := u_m(x_m) + \nabla u_m (x_m)\cdot (x-x_m)\}$$ has at least
two distinct components in $\overline{R}.$ One of the components
that does not contain the segment $\overline{R} \cap \{x_2=-1\}$
must intersect one of the lateral sides, say $x_1=\delta$ (for
infinitely many $m$'s.) Then we can find a polygonal line
connecting any neighborhood of $x_m$ with $x_1= \delta$ that is
included in this component.

Now, for each $t \in (0,\delta]$ we define $\underline{h}(t),
\overline{h}(t)$ to be the points in $\R^2$ on the line $x_1=t$
such that $$[\underline{h}(t), \overline{h}(t)]= \{u=0\} \cap
\{x_1=t\},
$$ (notice that $u$ is increasing in the $e_2$ direction.)
Let $c \in (a,b) \subset(0, \delta]$.

\

\textbf{Claim.} The segment $[\overline{h}(a), \overline{h}(b)]$
is above $\underline{h}(c)$ in the $e_2$ direction.

\

Indeed, assume by contradiction that our claim does not hold.
Then, there exists a linear function $l$ increasing in the $e_2 $
direction, such that
$$[\overline{h}(a), \overline{h}(b)] \subset \{l < -1\}, \quad \underline{h}(c) \in \{l>1\}.$$

Let $R_{ab} :=\{a < t < b , |x_2|< 1\}$. Denote by $U_m$ the
connected component of $\{u_m < l_m\}$ in $\overline{R_{ab}}$ that
contains the segment $\overline{R_{ab}} \cap \{x_2=-1\}.$ We
compare $u_m$ and $l_m+\eps l$ in the set $\overline{R_{ab}}
\setminus U_m,$ with $\eps$ small enough depending on $u$ and $l$.
We have
$$u-\eps l >0  \ \text{on} \ \overline{R_{ab}} \cap \{x_2=1\}$$ and
$$u-\eps l >0  \ \text{on} \ \{u\geq -\eps/2\} \cap (\{x_1=a\} \cup \{x_1=b\}).$$

Notice that, since $u_m \rightarrow u, l_m \rightarrow 0$
uniformly, we have that $$\{u < -\eps/2\} \cap
\overline{R_{ab}}\subset U_m$$ for all $m$ large enough, thus
$$u_m - l_m -\eps l
> 0 \quad \text{on $\partial R_{ab} \setminus U_m$.}$$ On the other
hand, $u$ is strictly negative on the segment $\{x_1=c\} \cap \{l
<0 \} \cap \overline{R_{ab}}$. Thus this segment is included in
$U_m$, which implies that $$u_m - l_m -\eps l <0 \quad \text{on} \
\p U \cap \{x_1=c\} \cap R_{ab}.$$

Hence the minimum of $u_m-l_m-\eps l$ in $\overline{R_{ab}}
\setminus U_m$ is negative and by the maximum principle it occurs
at some point $x_0 \in
\partial U_m \cap R_{ab}.$ Thus, \begin{equation}\label{minimum}u_m - l_m \geq \eps(l -l(x_0))
\quad \text{in $\overline{R_{ab}} \setminus U_m$.}
\end{equation}

Recall that there exists a polygonal line included in
$(\overline{R_{ab}} \setminus U_m) \cap \{u_m < l_m\}$ that
connects the lines $x_1=a$ and $x_2=b$. Now the right-hand side of
\eqref{minimum} is increasing in the $e_2$ direction and we obtain
a contradiction at a point where $x_0 + se_2, s\geq 0$ intersects
this polygonal line. Thus the claim is proved.

\

Next we prove that $\overline{h}(t)$ is a convex curve and
$\underline{h}(t)$ is a concave curve, if $t>0.$ Indeed, let $0 <
t_1< t_2 \leq \delta$, and let $Q$ be the convex set generated by
$\overline{h}(t_i), \underline{h}(t_i), i=1,2.$ From the claim
above we see that $$Q \cap [\underline{h}(t),\overline{h}(t)] \neq
\emptyset, \quad t_1 \leq t \leq t_2.
$$ Thus, $$u \leq 0 \quad \text{on} \ [\underline{h}(t_1),\underline{h}(t_2)].$$
Let $l'$ be the linear function that is $0$ on the line passing
through $\underline{h}(t_i) - \gamma e_2, i=1,2$ and has slope
$\eps$ in the $e_2$ direction. Clearly, if $\eps$ is small
depending on $u$ and $\gamma$, then $u \leq l'$ on the boundary of
$\{t_1 < x_1< t_2, x_2 \geq -1\} \cap \{l' \leq 0\}.$ Hence the
same inequality is true in the interior. Thus
$$u < 0, \quad \text{below} \
[\underline{h}(t_1),\underline{h}(t_2)] - \gamma e_2.$$ By letting
$\gamma$ tend to $0$ and repeating the same argument from above,
we find that $[\underline{h}(t), \overline{h}(t)] \subset Q.$

\end{proof}

Let $[a_1,a_2] \subset (p_1,p_2)$ with $e_2$ normal to $(p_1,p_2)$
and pointing inside $N$. For all $p \in [a_1,a_2]$ with the
property that for each $r>0$ there exists a sequence of $m
\rightarrow \infty$ such that
$$\nabla u_m(B_r) \cap \mathcal{B}_r(p)
\neq \emptyset,
$$ we define
\begin{equation}\label{Cp}
C_p:= \{u=p \cdot x\}.\end{equation}

Assume that $R:= \{|x_1|< \delta, |x_2| < 1\} \subset \Omega$ with
$$u> 0 \ \text{on} \ \overline{R} \cap \{x_2=1\}, \quad u< 0, \ \text{on} \  \overline{R} \cap \{x_2=-1\}, \quad
u(0)=0.$$ Then for each $p \in [a_1,a_2]$, possibly by taking
$\delta$ smaller, we also have $$u(x)> p\cdot x, \ \text{on} \
\overline{R} \cap \{ x_2=1\}, \quad u(x)< p\cdot x, \ \text{on} \
\overline{R} \cap \{ x_2=-1\}.
$$
Thus applying Lemma \ref{R} (with the origin replaced by $p$) for
each slope $p$ as above, either the set
$$C_p^+:= C_p \cap \{0<x_1 < \delta \}$$ or the set $$C^-_p: = C_p
\cap \{0< x_1 < \delta\}$$ is given by the region between a convex
function and a concave function. Let $A^+$ (resp. $A^-$) be the
set of $p \in [a_1,a_2]$ such $C_p^+ $ (resp. $C_p^-$) is given by
such region.

With this notation, we state and prove the following Lemma which
will be used for the proof of Proposition \ref{segments}.

\begin{lem}\label{defCp}Let $p \in A^+$ and $\alpha \in (0, \delta)$.
Then, for any neighborhoods $V$ of $C_p^+ \cap \{x_1=\alpha\}$ and
$W$ of $p$, there exist a point $\tilde{x} \in V$ such that
$\nabla u_m (\tilde{x}) \in W.$
\end{lem}

\begin{proof} Without loss of generality we assume $p=0 \in A^+$.
We refer to the proof of the previous lemma. We can assume that
$u_m$ is not linear, otherwise the statement is trivial. Let
$\alpha \in (a,b) \subset (0,\delta)$, and let us focus on the
connected component $\widetilde{U}_m$ of $\{u_m < l_m\}$ which
contains the polygonal line connecting $x_1=a$ with $x_1=b$. Since
$u_m \rightarrow u, l_m \rightarrow 0$ uniformly, we obtain that
$\widetilde{U}_m$ is in any neighborhood of $C_p$, if $m$ is large
enough. Now, consider the function $u_m-l_m
+\frac{1}{2}(x_1-\alpha)^2$. Then the minimum of this function in
$\widetilde{U}_m$ is negative and is achieved at an interior point
$\tilde{x}$ . The desired conclusion follows by taking the
interval $(a,b)$ sufficiently small.

\end{proof}

We are now ready to exhibit the proof of Proposition
\ref{segments}.

\

\noindent \textbf{Proof of Proposition \ref{segments}.} We take
$u(0)=0$ and use the notation of the previous Lemmas. Assume by
contradiction that $\{u=0\}$ does not contain either of the
segments in the direction $e_2$ connecting $0$ with $\partial
\Omega$. Then there exists $\delta$ small such that $R \Subset
\Omega,$ $u>0$ on $\overline{R} \cap \{x_2=1\}$ and $u<0$ on
$\overline{R} \cap \{x_2=-1\}.$ Therefore, given $[a_1,a_2]
\subset (p_1,p_2)$, we can define $A^+$ and $A^-$ as above (see
discussion before Lemma \ref{defCp}.)

Since $\overline{A^+} \cup \overline{A^-} = [a_1,a_2]$, then there
exists an open interval contained in either $\overline{A^+}$ or $
\overline{A^-}$, say $(-\rho e_1,\rho e_1) \subset
\overline{A^+}.$

Now set $$D_t:= \{u = te_1\cdot x\} \cap \{0<x_1<\delta\}, \quad
-\rho < t < \rho.
$$ Notice that if $te_1 \in A^+$ then $D_t$ is exactly $C_{te_1}^+.$

We claim that $D_t$ is precisely the region between two segments
with one end-point on $\{x_1=0\}$ and the other on
$\{x_1=\delta\}.$

Indeed, let $t_ne_1 \in A^+, t_n \uparrow t \in (-\rho,\rho),$
$t_n$ strictly increasing.  Notice that, if $x_1 >0$
\begin{equation}\label{Equal}\{u < tx_1\} = \bigcup_{n} \{u < t_n x_1\}= \bigcup_{n}
\{u \leq t_n x_1\}.\end{equation} Using that $C^+_{t_ne_1}$ is the
region between a convex and a concave graph, we obtain from
\eqref{Equal} that $\partial \{u < tx_1\}\cap \{0< x_1 < \delta\}$
is both a concave and a convex graph. This implies the claim.

Next, choose a point of differentiability $x^*$ for $u$ in the
open set $$\{-\rho e_1\cdot x < u < \rho e_1 \cdot x\} \cap
\{0<x_1<\delta\},$$ such that \begin{equation}\label{gamma}\nabla
u(x^*)\cdot e_2
>\gamma >0.\end{equation}

Set $p^*= \nabla u(x^*).$ Without loss of generality we can assume
$u(x^*)=0$ and thus the set $D_0$ consists of just one segment.
Therefore, if $t_n \uparrow 0$, $t_n$ increasing, $t_n e_1 \in
A^+$, then
\begin{equation}\label{converge}C_{t_ne_1}^+ \cap \{e_1 \cdot (x-x^*)= 0\} \rightarrow
x^*.\end{equation}

Since $x^*$ is a point of differentiability of $u$ and $u_m
\rightarrow u$ uniformly, we obtain that for any $\eps>0$ we can
find $\rho, m$ such that
$$|u_m - p^*\cdot (x-x^*)| \leq \eps \rho \quad \text{in} \
B_\rho(x^*).$$

On the other hand from \eqref{converge} together with Lemma
\ref{defCp}, we can find a point $x_m \in B_{\rho/2}$ such that
$$|\nabla u_m (x_m)| \leq \gamma/2.$$

Then, using \eqref{gamma} we see that the middle point of $[\nabla
u_m(x_m), p^*]$ is outside a $\gamma/4$ neighborhood of $\partial
N \cup \{p^*\}$ (provided that $\gamma$ is chosen small enough
depending on $N$). This contradicts either Proposition \ref{flat1}
(if $p^* \in N$) or Proposition \ref{flat2} (if $p^* \in \p N$),
by choosing $\eps=\eps(\gamma, p^*,N)$ sufficiently small.

 \qed

Now, we are finally ready to present the proof of Theorem 1.3.

\

\noindent \textbf{Proof of Theorem 1.3.} For simplicity take
$x_0=0$. Assume that there are no segments connecting $0$ with
$\partial \Omega$ as in the statement of the Theorem. We wish to
prove that given $\rho
>0$ there exists $\delta$ (depending also on $u$) small such that
$diam(\nabla u(B_\delta)) \leq \rho.$

By Proposition \ref{segments} for each side $[p_i,p_{i+1}]$ and
each subinterval $(a_i,a_{i+1}) \subset [p_i, p_{i+1}]$ there
exist a point $a \in (a_i,a_{i+1})$ and $r>0$ such that $$\nabla
u_m (B_r) \cap \mathcal{B}_r(a) = \emptyset.$$

Clearly we can find $\eps >0$ and points $a_k \in
\partial N$ such that any ball $\mathcal{B}_{\rho/2}(p)$ centered at $p \in \partial
N$ contains one of the balls $\mathcal{B}_{\eps}(a_k)$ and
$$\nabla u_m (B_\eps) \cap \mathcal{B}_\eps(a_k) = \emptyset.$$

We cover the open set $$O = \{x \in N | dist (x,\partial
N)>\eps/2\}$$ with a finite number of balls of radius $\eps/10$
centered in the set. According to the localization Theorem
\ref{localization}, if $\delta$ is small enough depending only on
$F,\eps$ and say $\mathcal{B}_\eps(a_1)$ we find that either
\begin{equation}\label{d1}diam (\nabla u_m(B_\delta)) \leq \eps/5\end{equation} or \begin{equation}\label{n2}\nabla
u_m (B_\delta) \cap O =\emptyset.\end{equation} In the latter case
we know from Proposition \ref{compact} that for $m$ large enough
\begin{equation}\label{n3}\nabla u_m (B_\delta) \ \text{is included in a
$\eps/2$-neighborhood of $N$}. \end{equation} Since $\nabla u_m
(B_\delta)$ is connected, we conclude from the choice of $\eps,$
\eqref{n2} and \eqref{n3} that $$diam(\nabla u_m( B_\delta)) \leq
\rho.$$ The conclusion follows by letting $m \rightarrow \infty.$

\qed

We conclude this Section with the proof of Theorem \ref{convex}.

\

\noindent \textbf{Proof of Theorem \ref{convex}.} For simplicity
assume $p_i=0$, $c=0$ and $\omega=e_2$. Since $u$ is increasing in
a cone of directions around $e_2$ we conclude that $S$ is between
two Lipschitz graphs in the $e_2$ direction. Assume by
contradiction that the lower graph is not concave. Then we can
find two points $z_1$, $z_2$ such that $u(z_i)<0$ and the segment
$[z_1,z_2]$ is tangent from below to $S$. Let $l$ be the linear
function which vanishes at the points $z_i+\eps e_2$ with $\nabla
l \cdot e_2=\eps$. We compare $u$ and $l$ in the set $\{x\cdot z_1
< x_1 < x \cdot z_2\} \cap \{l < 0\} \cap \Omega$. If $\eps$ is
small enough then $u \le l$ on the boundary, hence $u \le l$ in
this set. This implies that $S$ is above the segment $
[z_1,z_2]+\eps e_2$ and we contradict the fact that $[z_1,z_2]$
was tangent to $S$.

\qed

\section{The proof of Theorem \ref{continuous}}

We distinguish two cases, when the polygon $N$ has $n\geq 4$
vertices and when it has only $n=3$ vertices. The latter is more
involved and we only present a sketch of the proof.

\

\begin{lem}Assume $N$ has more that $3$ vertices. Then $H(\nabla u)$ is continuous.\end{lem}
\begin{proof}If $H(\nabla u)$ is not continuous at the origin,
then for each side $[p_i,p_{i+1}]$ there exists a segment of
direction $\nu_i$ perpendicular to that side, starting from the
origin  and ending at $x_i \in \partial \Omega$ such that
$u(x)=u(0)+ p_i\cdot x, x \in [0,x_i].$ This follows from the
localization Theorem and from Proposition \ref{segments}.

For simplicity of exposition, assume that $$p_1= |p_1|(-\cos
\theta, \sin \theta), \quad p_2 = 0, \quad p_3= |p_3|(\cos \theta,
\sin \theta), \quad \theta \in (0,\pi/2).
$$ Recall that in our notation $\nu_i$ is either of the two directions perpendicular to $[p_i, p_{i+1}]$.
Let us choose $\nu_1,\nu_2$ pointing inside $N$, that is
$$\nu_1= (\sin \theta, \cos \theta) , \quad \nu_2= (-\sin \theta, \cos \theta).$$

Assume $u(0)=0$, then $u = 0$ on two segments starting at 0 of
directions $\pm \nu_1$ and $\pm \nu_2$ respectively. Denote by
$\vartheta_2$ the closed angle (smaller than $\pi$) generated by
these two directions. We wish to prove that
\begin{equation}\label{Limit}\lim_{x\rightarrow 0, x \in \vartheta_2} H(\nabla
u)=H(0),\end{equation} and moreover that no segment of direction
$\nu_j, j \neq 1,2$ along which $u= p_j \cdot x$ can intersect the
angle $\vartheta_2$.

We distinguish two cases.

\

 \textit{Case 1.} $u =0$ along segments of directions $\nu_i$, $i=1,2$ (or analogously $-\nu_i$, $i=1,2$.)

Since $\nabla u \in \overline{N}$, then $u$ is increasing in the
cone of directions in $\vartheta_2$, thus $u \geq 0$ in
$\vartheta_2$. We also have that $u-p_1\cdot x$ is increasing in
the direction $(\sin(\theta +\delta),\cos (\theta+\delta))$ with
$\delta>0$ small depending on $N$. This implies that near the
origin in $\vartheta_2$ we have
 that $u \leq p_1\cdot x,$  and analogously $u \leq p_3 \cdot x.$ Thus, in $B_\rho \cap
 \vartheta_2$ we have $$0 \leq u \leq \min\{p_1\cdot x, p_3 \cdot
 x\}.$$

 We use the hypothesis $n \ge 4$ to prove that no segment of direction $\nu_j, j \neq 1,2$ in
 which $u=p_j \cdot x$ can intersect the angle $\vartheta_2$.
 Indeed, in this angle near the origin we have $u \leq p \cdot x$ for all $p \in
 [p_1,p_3]$ and also $p\cdot x >0$ when $p \in (p_1,p_3)$. If $u = p_j \cdot x$  on a segment in the interior of
 the angle $\vartheta_2$
 then we reach a contradiction since there exists $0<\lambda<1$ such that $\lambda p_j \in
(p_1,p_3)$ and on this segment $u > \lambda p_j \cdot x \geq u$.

In $B_\rho$ we consider the function $\tilde{u}$ such that
$\tilde{u}=u$ in the angle $\vartheta_2$ and $\tilde{u}=0$ outside
this angle. Then, $\tilde{u}$ is a minimizer for $I$ in $B_\rho$.
This follows from the fact that the gradient of the minimizer with
boundary data $\tilde{u}$ on $\p B_\rho$ belongs to $\overline N$
and therefore the minimizer must be bounded by 0 from below and by
$\max\{p_i \cdot x, 0\}, i=1,3$ from above.

Clearly if $j \geq 3$, then $p_j \cdot \nu_j \neq 0$ while
$\tilde{u}=0$ outside the angle $\vartheta_2$. This implies that
$\tilde{u} \neq p_j \cdot x,$ on a segment of direction $\nu_j$.
Thus we conclude from the localization Theorem and from
Proposition \ref{segments} that $H(\nabla \tilde{u})$ is
continuous and
 $$\lim_{x\rightarrow 0, x \in \vartheta_2} H(\nabla u)=\lim_{x\rightarrow 0} H(\nabla \tilde{u})=H(0).$$

\

 \textit{Case 2.} $u =0 $ along segments of directions  $\nu_1$, and $-\nu_2$ (or analogously $-\nu_1,\nu_2$.)

Then $u = 0$ in $B_\rho \cap \vartheta_2$ and \eqref{Limit}
clearly holds. Also, as in the conclusion of the previous case no
segment of direction $\nu_j$, $j \geq 3$ in which $u=p_j \cdot x$
can intersect $\vartheta_2$.

\

 Since the angles generated by consecutive directions cannot
 overlap, they must cover a neighborhood of the origin and the
 lemma is proved.

\end{proof}

\begin{lem} If $N$ has three sides, $H(\nabla u)$ is continuous.
\end{lem}

\noindent \textit{Sketch of the proof.} We use the same notation
as in the previous proof. The only case that does not follow from
a similar analysis as before, is when $u$ is linear on segments of
directions $\nu_1,\nu_2$ and $-\nu_3$ ($\nu_i$ points inside $N$).
More precisely, $u=\min\{p_1 \cdot x, p_3 \cdot x\}$ in $B_1 \cap
\vartheta_2,$ and $u \leq 0$ outside $\vartheta_2$. We can also
assume that $u$ is not linear in the direction of $\nu_3$. We can
assume further that $u$ is not linear on the segments of direction
$-\nu_1$ (or $-\nu_2$). Otherwise, as in the proof of the previous
Lemma, we can construct another minimizer $\tilde{u}=\min\{u, p_1
\cdot x \}$ which coincides with $u$ on $\{p_1 \cdot x \leq 0\}.$
The continuity of $H(\nabla \tilde{u})$ (and hence of $H(\nabla u
)$) follows as in the case $n \geq 4.$

In the remaining case we wish to prove that the assumption
\eqref{hardassumption} of the localization Theorem holds, hence
the Lemma follows.

First, suppose that there exist a direction $p \in (p_1,p_3)$ and
a rectangle \begin{equation}\label{rectangle}R_{ab}= \{a<x_1<b,
|x_2|<1/2\}, \quad 0\in (a,b) \subset [-1/2,1/2]\end{equation}
such that \begin{equation}\label{uneg}u <0 \ \text{on} \
\partial R_{ab} \cap \{p \cdot x \leq 0\}.\end{equation} Then, for small $\eps$ (depending on $u$) any linear
function $l$ with $$\nabla l \in \mathcal{B}_{\eps^2}(\eps p),
\quad |l(0)|\leq \eps^2$$ has the property that $\{u <l\} \cap
\partial R_{ab}$ consists of one connected component. This
implies that if we approximate $u$ in $R_{ab}$ by functions $u_m$
($u_m=u$ on $\p R_{ab}$), then $$\nabla u_m(B_{\eps^3}) \cap
\mathcal{B}_{\eps^2}(\eps p)=\emptyset,$$ otherwise we contradict
Theorem \ref{components}. Hence the Lemma follows from the
localization Theorem.

The only case when we cannot find one pair $p, (a, b)$ satisfying
\eqref{rectangle}-\eqref{uneg}, is when $u = 0$ in $B_{\rho}$
($\rho$ small depending on $u$) above a line passing through the
origin and below the angle $\vartheta_2$ (recall that $u$ is
increasing in the $e_2$ direction). After a dilation, assume
$\rho=1$. Therefore, we can assume that
$$0 \geq u \geq \tilde{p}\cdot x, \quad \text{in} \ B^-:=B_{1}
\cap \{\tilde{p}\cdot x \leq 0\},
$$ for some $\tilde{p} \in (p_1,p_3)$.

Now one can obtain a localization Theorem in $B^-$ around the
origin, even tough $0$ is a boundary point. Indeed, let $u_m$ be
the approximation for $u$ in $B^-$ ($u_m=u$ on $\p B^-$). Then
$\nabla u_m(x) \in [0,\tilde{p}] $ for all $x \in \p B^- \cap
\{\tilde{p} \cdot x =0\}$. Thus, we can obtain a Caccioppoli-type
inequality at the boundary as in the proof of the localization
Theorem, as long as the function $\eta$ in formula \eqref{Step2}
(the domain of integration is now $B^-$) is 0 on the segment
$[0,\tilde{p}]$ and $\xi \in C^\infty_0(B_{1}).$ This implies that
the iteration argument in the proof of the localization Theorem is
valid provided that $\mathcal{B}_{4r}^k \cap
[0,\tilde{p}]=\emptyset.$

Also, a boundary version of Proposition \ref{segments} holds
because the ``thin'' connected components of $\{u_m < l_m\}$
cannot intersect $\tilde{p} \cdot x =0.$ More precisely, given
$a_i \in \partial N \setminus \{\tilde{p}\}$, one can show that
there exists $r$ depending on $u$ and $a_i$ such that
$$\nabla u_m(B_r \cap B^-)\cap
\mathcal{B}_{r}(a_i) =\emptyset.$$

Thus for any $\delta$, there exists $\eps$ depending on $\delta$
and $u$ such that $\nabla u_m(B_\eps \cap B^-)$ is in a $\delta$
neighborhood of the segment $[0,\tilde{p}].$ Letting $m
\rightarrow \infty$ we obtain the same result for $u.$

Consider now a sequence of blow-up minimizers $$\frac{1}{r_k}
u(r_k x), x \in B_1,$$ that converges to $\overline{u}.$ Clearly,
$\overline{u}$ is still a minimizer for $I$ and by the conclusion
above $\nabla \overline{u}(B^-) \subset [0,\tilde{p}].$ This
implies that $\overline{u}(x)=p \cdot x$ in $B^-$ with $p \in
[0,\tilde{p}]$. If $p \neq 0$, then we reach a contradiction since
$\overline{u}$ is not a minimizer in a neighborhood of points $x$
such that $\tilde{p} \cdot x =0, x \neq 0.$

In conclusion for any $\eps>0$ there exists $r_\eps$ depending on
$u$ and $\eps$ such that $$|\frac{1}{r}u(rx)|\leq \eps, \quad x
\in B^-$$ for all $r \leq r_\eps.$ Now the result follows applying
the flatness Theorem \ref{Flat2} in balls $B_{c|x|}(x)$, for $c$
small enough and $x \in B^-$.

\qed

\section{The proof of the flatness Theorems}

We finally present the proofs of Theorem \ref{Flat1} and Theorem
\ref{Flat2}. Since $u_m \rightarrow u$ uniformly, it suffices to
prove Proposition \ref{flat1} and Proposition \ref{flat2}.

\

\noindent\textbf{Proposition 6.2.} \textit{Assume $B_1 \subset
\Omega$ and
$$\mathcal{B}_{\delta}(p_0) \subset N.$$Then, there exists $\eps$
depending on $\delta, p_0, F,$ such that if $$|u_m - p_0 \cdot
x|\leq \eps, \quad x \in B_1$$ then
$$\nabla u_m (B_{1/2}) \subset \mathcal{B}_{\delta}(p_0).
$$}

\noindent \textit{Proof.} We can assume that $u_m$ is not linear,
otherwise the result is obvious. Also, it suffices to show that
the conclusion of our statement holds at 0, i.e.
$$\nabla u_m (0) \in \mathcal{B}_{\delta}(p_0), \quad \text{for large
$m$.}$$

First notice that \begin{equation}\label{star}\nabla
u_m(B_{2\sqrt{\eps}}(0)) \cap \mathcal{B}_{2\sqrt{\eps}}(p_0) \neq
\emptyset\end{equation} by considering $\min\{\frac{1}{2}|x|^2 +
u_m - p_0 \cdot x \}.$

Let $\mathcal{B}_{2\rho}(p_1) \subset N$ such that $F \in
C^2(\overline{\mathcal{B}_{2\rho}(p_1)}),$ $p_0$ not in
$\mathcal{B}_{2\rho}(p_1)$. We claim that
\begin{equation}\label{claim}\nabla u_m(B_{1/2}) \cap
\mathcal{B}_{\rho}(p_1)=\emptyset, \quad \text{for small $\eps$.}
\end{equation}
Clearly, the proposition follows from the claim, together with
\eqref{star} and the localization Theorem. We are left with the
proof of the claim.

Assume by contradiction that there exists $x_0 \in B_{1/2}$ such
that $$p_2 := \nabla u_m (x_0) \in \mathcal{B}_{\rho}(p_1).$$ We
know that the set
$$\{u_m(x) < l_m: = u_m(x_0)+ p_2\cdot (x-x_0)\}$$ has in
$B_1$ at least two distinct connected components that intersect
any neighborhood of $x_0$.

Notice that by the flatness assumption

$$\{(x-x_0) \cdot (p_0-p_2) < -2\eps\} \subset \{u_m < l_m\}$$
and
$$\{(x-x_0) \cdot (p_0-p_2) > 2\eps\} \subset \{u_m > l_m\}.$$

This implies that one of the connected components of $\{u_m <
l_m\}$ is included in the strip $\{|(x-x_0)\cdot (p_0 -p_2)| \leq
2\eps \}.$

By changing the system of coordinates in the $x$ and $p$ spaces,
we can assume that we have the following situation:
$$p_2 =0, \quad p_0=\alpha e_2, \quad x_0=\pm e_1/4, \quad u_m(x_0)=0$$
$$|u_m - \alpha e_2\cdot x|\leq 2\eps \quad \text{in} \ \overline{R}:= \{|x_1|\leq 1/8, |x_2|\leq 1/8\}, \ \alpha > \rho.$$
Moreover, the set $\{u_m <0\}$ has one connected component in
$\overline{R}$ included in the strip $\{|x \cdot \alpha e_2| \leq
2\eps \}$ that intersects both $\{x_1=\pm1/8\}.$

Let $U$ be the connected component that contains $\overline{R}
\cap \{x_2=-1/8\}.$ In $\overline{R}\setminus U$ we compare $u_m$
with the function $w$ defined by
$$w = \delta' g(v), \quad v(x)= x_2-20x_1^2$$
$$g(v)=e^{kv}-constant, \quad g(-1/8)=0,$$ for some $\delta', k$
to be chosen later.

Now notice that $D^2F_m$ is uniformly elliptic in
$\mathcal{B}_\rho$ with ellipticity constants $\lambda, \Lambda$
depending only on $F.$ Since $$D^2w= \delta' e^{kv}(D^2v + k
\nabla v \otimes \nabla v)$$ we see that if the constant $\delta'$
is chosen small enough so that $\nabla w \in \mathcal{B}_\rho,$
and $k$ is sufficiently large depending only on $\lambda, \Lambda$
we have that
$$Tr(D^2F_mD^2w) \geq \lambda |(D^2w)^+| - \Lambda |(D^2w)^-|>0.$$
Therefore, $w$ is a subsolution and the minimum of $u_m-w$ must
occur on $\partial (\overline{R}\setminus U).$

Notice that by choosing $\delta'$ possibly smaller depending on
$\rho$ and $k$, we get that
$$ \alpha e_2\cdot x  > w \quad \text{on} \ \{x_2=1/8, |x_1|\leq 1/8\} \cup \{x_1=\pm 1/8 , x_2 \geq
-\delta''\},$$ and therefore
$$u_m -w  >0 \quad \text{on} \ \{x_2=1/8, |x_1|\leq 1/8\} \cup \{x_1=\pm 1/8 , x_2 \geq
-2\eps/\rho\},$$ for small $\eps.$ However, $u_m - w <0$ on
$\partial U \cap \{x_1=0\}$. Hence the minimum must occur at some
point $z_0 \in \partial U \cap R.$ Then, $$u_m(x) \geq w(x)
-w(z_0), \quad x \in \overline{R} \setminus U.$$ As in Proposition
\ref{segments}, this is a contradiction since the line $z_0 +
te_2, t\geq 0$ intersects the other connected component of $\{u_m
<0\}$, while on this line the function $w$ is increasing.

\qed

\

\noindent \textbf{Proposition 6.3.}\textit{ Let $p_0 \in \partial
N.$ Then, there exists $\eps$ depending on $\delta, F$, such that
if
$$|u_m - p_0\cdot x| \leq \eps, \quad x \in B_1$$ then
$$[\nabla u_m(x), p_0] \quad \text{is in a $\delta$-neighborhood of $\partial N$, for all $x \in
B_{1/2}.$}$$}

\noindent\textit{Proof.} We can argue as in the previous proof.
Notice that in the proof of \eqref{star} and \eqref{claim}, we do
not use that $p_0 \in N$, thus they hold also for $p_0 \in
\partial N$. Then, when applying the localization Theorem, using also
Proposition \ref{compact}, we conclude that $\nabla u_m (B_{3/4})$
is included in a $\delta$-neighborhood $\mathcal{N}_{\delta}$ of
$\partial N$ if $\eps$ is small depending only on $\delta$, $F$.

Again, it suffices to prove our conclusion at 0, that is $[\nabla
u_m(0), p_0] \subset \mathcal{N}_{4\delta}.$ By changing the
system of coordinates in the $x$ and $p$ spaces, we can assume
that we have the following situation:
$$\nabla u_m(x_0) =0, \quad u_m(x_0)=0, \quad x_0=\pm e_1/4, \quad p_0=\alpha e_2$$
$$|u_m - \alpha e_2\cdot x|\leq 2\eps \quad \text{in} \
\overline{R}:= \{|x_1|\leq 1/8, |x_2|\leq 1/8\},$$ and $\nabla
u_m(\overline R) \subset \mathcal{N}_\delta$. We need to show that
$[0,p_0] \subset \mathcal{N}_{4\delta}$. We can assume that
$\alpha
>\delta$, otherwise the conclusion clearly holds since $p_0 \in \p N$. Moreover, the set $\{u_m <0\}$ has one connected component
in $\overline{R}$ included in the strip $\{|x \cdot \alpha e_2|
\leq 2\eps \}$ that intersects both sides $\{x_1=\pm1/8\}.$

In the set $\overline{R}\setminus U$ ($U$ as in the previous
proof) we compare $u_m$ with the function $w$ given by
$$w(x)=\frac{\alpha}{2}(x_2 -20x_1^2) +2\eps.$$
For small $\eps$ we get
$$u_m -w  >0 \quad \text{on} \ \{x_2=1/8, |x_1|\leq 1/8\} \cup (\{x_1=\pm 1/8\}\cap (\overline{R}\setminus U)).$$
However, $u_m - w <0$ on $\partial U \cap \{x_1=0\}$. Hence the
minimum of $u_m-w$ is negative and must occur at some point $z_0
\in R \setminus U.$ On the other hand, since $w$ is increasing in
the $x_2$ direction, we obtain as in the previous proof that the
minimum cannot occur on $R \cap
\partial U$. Thus $z_0$ is an interior point and $$\nabla u_m(z_0)
= \nabla w (z_0).$$

Moreover $|z_0| \leq \sqrt{\eps/\alpha}$. Indeed, it is
straightforward to check that if
$$x \in \{\alpha x_2 \geq -2\eps,|x| \geq
\sqrt{\eps/\alpha}\} \supset (\overline{R} \setminus U) \setminus
B_{\sqrt{\eps/\alpha}}(0),$$ then
$$u_m -w \ge \alpha x_2-w -2\eps
\ge \eps,$$ and the claim follows. Therefore, as $\eps \rightarrow
0$, $\nabla u_m(z_0) \rightarrow \alpha e_2/2.$ If $\eps$ is small
enough (depending on $\delta$ and $F$), $\nabla u_m(z_0)\in
\mathcal {B}_{\delta}(p_0/2)$ and $ \nabla u_m(z_0) \in \mathcal
{N}_\delta$ since $z_0 \in \overline {R}$. Hence $p_0/2 \in
\mathcal{N}_{2 \delta}$ which implies $[0,p_0] \subset
\mathcal{N}_{4\delta}$, and the proposition is proved.

\qed

\end{document}